\documentclass[12pt]{amsart}
\usepackage{amssymb,amsthm,amsmath,amstext}
\usepackage{mathrsfs} 
\usepackage{bm} 

\usepackage{relsize}

\usepackage[left=1.28in,top=.8in,right=1.28in,bottom=.8in]{geometry}

\usepackage{graphicx}
\usepackage[all]{xy}

\theoremstyle{plain}
\newtheorem{theorem}{Theorem}[section]
\newtheorem{prop}[theorem]{Proposition}
\newtheorem{lemma}[theorem]{Lemma}
\newtheorem{notation}[theorem]{Notation}

\newtheorem{cor}[theorem]{Corollary}

\theoremstyle{definition}
\newtheorem{defin}[theorem]{Definition}

\theoremstyle{remark}
\newtheorem{remark}[theorem]{Remark}
\usepackage{hyperref}
\usepackage{color}

\newcommand{\PP}{\mathcal{P}}
\newcommand{\XX}{\mathscr{X}}
\newcommand{\Z}{\mathbf{Z}}

 \DeclareFontFamily{U}{wncy}{}
 \DeclareFontShape{U}{wncy}{m}{n}{<->wncyr10}{}
 \DeclareSymbolFont{mcy}{U}{wncy}{m}{n}
 \DeclareMathSymbol{\Sh}{\mathord}{mcy}{"58} 
 
\hypersetup{colorlinks=true,linkcolor=blue,anchorcolor=blue,citecolor=blue,letterpaper=true}

\begin{document}

\title[Local-global principles]
{Local-global principles for norm one tori over semi-global fields}

\author[]{ Sumit Chandra Mishra }
\address{Department of Mathematics, 
Emory University, 
400 Dowman Drive NE, 
Atlanta, GA 30322, USA}
\email{sumitcmishra$@$gmail.com}

\date{}

\begin{abstract}
Let $K$ be a complete discretely valued field with 
the residue field $\kappa$.
Let $F$ be the function field of a smooth, projective, 
geometrically integral curve over $K$ 
and $\XX$ be a regular proper model of $F$ such that 
the reduced special fibre $X$ is a union of regular curves 
with normal crossings.
Suppose that the graph associated to 
$\XX$ is a tree (e.g. $F = K(t)$).
Let $L/F$ be a Galois extension of degree $n$ such that 
$n$ is coprime to char$(\kappa)$. 
Suppose that $\kappa$ is an algebraically closed field or 
a finite field  containing a primitive $n^{\rm th}$ root of unity. 
Then we show that the local-global principle holds for the 
norm one torus associated to the extension $L/F$ 
with respect to discrete valuations on $F$ i.e. 
an element in $F^{\times}$ is a norm 
from the extension $L/F$ if and only if 
it is a norm from the 
extensions $L\otimes_F F_\nu/F_\nu$ 
for all discrete valuations $\nu$ of $F$.
\end{abstract}
 
\maketitle

\section{Introduction}

Let $F$ be a field and $\Omega_F$ be 
the set of all discrete valuations on $F$.
For $\nu \in \Omega_F,$ let $F_\nu$ 
denote the completion of $F$ at $\nu$.
Let $G$ be a linear algebraic group over $F$.
One says that the {\it local-global principle} holds 
for $G$ if 
for any $G$-torsor $X$,  $X$ has a rational point over 
$F$ if and only if it has a rational point over 
$F_\nu$ for all $\nu \in \Omega_F$. 
If $F$ is a number field, we also consider 
the completions at archimedean places 
while discussing local-global
principles for algebraic groups. 
If $F$ is a number field, then it is known that 
the local-global principle holds for various  classes of 
linear algebraic groups (\cite[Chapter 6]{PRAGNT}), 
including semisimple simply connected groups.
In particular, it is well-known 
that if $T_{L/F}$ is the norm one torus 
associated to a cyclic extension $L/F,$ 
then the local-global principle holds for 
$T_{L/F}$ i.e. an element $\lambda \in F^{\times}$ 
is a norm from the extension $L/F$ if and only if 
$\lambda$ is a norm from $F\otimes_F F_\nu/F_\nu$ 
for all $\nu \in \Omega_F$ (\cite[Chapter 11]{CFANT}). 
However, very little is known for general fields.

Let $K$ be a complete discretely valued field 
with residue field $\kappa.$ 
Let $F$ be the function field of a smooth, projective, 
geometrically integral curve over $K.$ 
Such a field $F$ is called semi-global field.
Let $G$ be a linear algebraic group over $F$. 
Harbater, Hartmann and Krashen (\cite{HHK1}) 
developed patching techniques to study 
$G$-torsors over $F$ and proved 
that if $G$ is connected and $F$-rational, 
then a $G$-torsor over $F$ 
has a rational point over $F$ if and only if it has a 
rational point over certain overfields of $F$ 
which are defined using patching (see Subsection \ref{prelim_patching}). 
As a consequence of this result, 
Colliot-Th\'el\`ene, Parimala and Suresh 
(\cite[Theorem 4.3.]{CTPS}) showed that if 
$G$ is reductive, $F$-rational and 
defined over the ring of integers of $K,$ 
then the local-global principle holds for $G$. 
Similar local-global principles are proved for 
various linear algebraic groups $G$ over $F$ if 
the residue field of $K$ is either finite or 
algebraically closed field (\cite{CTPS}, 
\cite{CTPS1},  \cite{Hu}, \cite{P}, \cite{PPS}). 

The first example of a linear algebraic group $G$ over $F$ 
where such a local-global principle fails was given by 
Colliot-Th\'el\`ene, Parimala and Suresh 
(\cite[Section 3.1. \& Proposition 5.9.]{CTPS1}). 
In their example, the residue field of $K$ is the field of 
complex numbers, $G$ is the norm one torus of a 
Galois extension $L/F$ with Galois group 
$\Z/2\Z \times \Z/2\Z$ and the field $F$ has a 
regular proper model with the associated graph not a tree.
Suppose that $F$ has a regular proper model 
with the associated graph a tree.
If $L/F$ is a Galois extension with Galois group 
$\Z/2\Z \times \Z/2\Z$ and $\kappa$ is algebraically closed, 
then they also proved that the local-global principle holds for 
the norm one torus $T_{L/F}$ 
(\cite[Section 3.1. \& Corollary 6.2.]{CTPS1}).\\

The main aim of this paper is to 
 prove the following theorem 
 (see Corollary \ref{sha_m_local}):
 
 \vspace{0.1in}
 
\begin{theorem} 
 Let $K$ be a complete discretely valued field with 
 residue field $\kappa$ and $F$ be the function field 
 of a smooth, projective, geometrically integral 
 curve over $K.$
Let $\XX$ be a regular proper model of $F$ with 
reduced special fibre $X$ a union of regular curves 
with normal crossings. 
 Let $L/F$ be a Galois extension over $F$ of degree $n$ 
 with Galois group $G$. Suppose that the graph associated to $\XX$ is a tree 
 and $\kappa$ is one of the following:\\
 
 $\bullet$ $\kappa$ is an algebraically closed field of characteristic 
 coprime to $n$, or\\
 
 $\bullet$ $\kappa$ is a finite field of characteristic 
 coprime to $n$ and contains a primitive $n^{\rm th}$ root of unity.\\

Then the local-global principle holds for the norm 
 one torus $T_{L/F}$ i.e. an element $\lambda \in F$ 
 is a norm from the extension $L/F$ if and only if $\lambda$ is a 
 norm from the extensions $L \otimes_F F_\nu/F_\nu$ 
 for all $\nu \in \Omega_F$. 
 \end{theorem}

\vspace{0.1in}
 For a finite separable extension $L/F,$ 
 let $T_{L/F}$ denote the norm one 
 torus associated to $L/F$. 
 For any extension $N/F,$ let $RT_{L/F}(N)$ be the subgroup of 
 $T_{L/F}(N)$ consisting of $R$-trivial elements (see Subsection \ref{Requivalence}). 
 The above theorem follows from the following more 
 general theorem (\ref{dvrsha}), where we allow more 
 general residue fields $\kappa$:
 
 \vspace{0.1in}

\begin{theorem} 
\label{main_theorem}
 Let $K$ be a complete discretely valued field with 
 residue field $\kappa $ and $F$ be 
 the function field of a smooth, projective, 
geometrically integral curve over $K.$ 
Let $\XX_0$ be a regular proper model of $F$ with 
 reduced special fibre $X_0$  a union of regular curves with normal crossings. 
 Let $L/F$ be a Galois extension over $F$ of degree $n.$ 
 Suppose that the graph associated to $\XX_0$ is a tree and:\\
 
 $\bullet$ char$(\kappa)$ is coprime to $n,$ \\
 
 $\bullet$ $\kappa$ contains a primitive $n^{\rm th}$ root of unity $\rho,$ and \\
 
 $\bullet$ for all finite extensions $\kappa '/ \kappa $ and 
 for all finite Galois extensions $l/\kappa '$ of degree $d$ dividing $n,$ 
 $$T_{l/\kappa '}(\kappa ') = RT_{l/\kappa '}(\kappa ') <\rho^{\frac{n}{d}}>.$$\\ 
 
 Then the local-global principle holds for the norm one torus $T_{L/F}$.
\end{theorem}

\textbf{Remark}: In fact one can restrict to 
divisorial discrete valuations 
 in the above theorems.\\ 

In the last section of the paper, we also give 
counterexamples to the local-global principle 
for certain norm one tori and 
multinorm tori over a semi-global field.

We now briefly describe the strategy of the proof 
of Theorem \ref{main_theorem}. 
Let $K$ be a complete discretely valued field 
with residue field $\kappa$ 
and $F$ be 
the function field of a smooth, projective, 
geometrically integral curve over $K.$
Let $\XX$ be a regular proper model of $F$ with 
reduced special fibre $X$ a union of regular curves 
with normal crossings. 
For any point $P\in X,$ let $F_P$ be the fraction field 
of the completion of the local ring at $P$ on $\XX$. 

For a linear algebraic group $G$ over $F$, 
let us define:\\
 $$ \bullet \displaystyle \Sh_X(F, G) := \text{ker}\left( H^1(F, G) \to  \prod_{P \in X} H^1(F_P, G) 
\right) $$ 
$$  \bullet \displaystyle \Sh(F, G) := \text{ker}\left ( H^1(F, G) \to  \prod_{\nu  \in \Omega_F} H^1(F_\nu, G) \right).$$ 
It is known that $\Sh_X(F, G) \subseteq \Sh(F, G)$ 
(\cite[Proposition 8.2.]{HHK1}).\\

In general, $\Sh_X(F, G)$ and $\Sh(F, G)$ 
are just pointed sets, but they are abelian groups if $G$ is abelian. 
The pointed sets $\Sh_X(F, G)$ and $\Sh(F, G)$ 
measure the obstruction to the local-global principle for 
the group $G$ with respect to points on $X$ and 
with respect to discrete valuations on $F$ respectively.\\

First we prove the following (\ref{Sha_X_vanishes}):

\vspace{0.1in}

\begin{theorem} 
\label{int-sha-x}
 Let $K$ be a complete discretely valued field 
 with residue field $\kappa $ and 
 $F$ be 
 the function field of a smooth, projective, 
geometrically integral curve over $K.$ 
 Let $\XX$ be a regular proper model of $F$ with 
 reduced special fibre $X$ 
 a union of regular curves with normal crossings. 
 Let $L/F$ be a Galois extension over $F$ of degree $n.$ 
 Suppose that \\
 
 $\bullet$ $n$ is coprime to char$(\kappa),$\\
 
 $\bullet$ $\kappa$ contains a primitive $n^{\rm th}$ 
 root of unity $\rho,$ \\
 
 $\bullet$ for all finite extensions $\kappa '/ \kappa$ 
 and for all finite Galois extensions $l/\kappa '$ of 
 degree $d$ dividing $n,$ 
$$T_{l/\kappa '}(\kappa ') = 
RT_{l/\kappa '}(\kappa ') <\rho^{\frac{n}{d}}>,$$ \\

$\bullet$ the graph associated to $\XX$ is a tree.\\ 

Then $\Sh _{X}(F, T_{L/F}) = 0.$ 
\end{theorem}

We conclude our main theorem (Theorem \ref{main_theorem}), 
by proving that for $K$, $F$ and $L$ as in 
Theorem \ref{int-sha-x}, 
$\displaystyle \Sh(F, T_{L/F}) = \bigcup_X \Sh_X(F, G)$ 
(\ref{union_of_Sha_X_equals_Sha_dvr}), 
where $X$ is  running over the reduced special fibres 
of regular proper models $\XX$ of $F$ which are 
obtained as a sequence of blow-ups of 
$\XX_0$ centered at closed points of $\XX_0$.

\vspace{0.1in}

\section{Preliminaries}

In this section, we recall a few basic definitions and facts 
about patching and $R$-equivalence on algebraic groups 
(\cite{CTS}, \cite{G}, \cite{HHK}, \cite{HHK1}) which will 
be used in this paper. 

\vspace{0.1in}

\subsection{Patching and various Shas} \hfill 
\label{prelim_patching} 

Let $F$ be a field and $\Omega _{F}$ be the set of all 
equivalence classes of discrete valuations $\nu$ on $F.$ 
For $\nu \in \Omega _{F},$ let $F_{\nu }$ denote 
the completion of $F$ at $\nu$ and $\kappa(\nu)$ 
the residue field at $\nu.$ 
For an algebraic group $G$ over $F,$ let
 $$\Sh (F, G) := \text{ker} \left( H^1(F, G) \rightarrow 
\displaystyle \prod_{\nu \in \Omega_{F}} H^1(F_{\nu }, G) \right).$$

In this paper, we are concerned with 
a special class of fields called semi-global fields.
\vspace{0.1in}

\begin{defin}{(Semi-global field)}
\label{SGF}
A {\it semi-global field} is the function field of 
a smooth, projective, geometrically integral 
curve over a complete discretely valued field.
\end{defin}

\noindent \textbf{Example:} Some examples of semi-global fields 
are: $\mathbb{C}((t))(x)$, $\mathbb{Q}_p(x)$, $\mathbb{F}_p((t))(x)$ and \\
$\displaystyle \frac{\mathbb{C}((t))(x)[y]}{\langle xy(x+y-1)-t \rangle}$.\\

Let $T$ be a complete discretely valued ring with 
fraction field $K$ and residue field $\kappa$. 
Let $t \in T$ be a parameter.  
Let $F$ be a function field of a smooth, 
projective, geometrically integral curve over $K.$ 
Thus $F$ is a semi-global field.
Then there exists a regular $2$-dimensional integral 
scheme $\XX$ which is proper over $T$ with function 
field $F.$ We call such a scheme $\XX$ a 
{\it regular proper model} of $F.$ 
Further there exists a regular proper model 
of $F$ with the reduced special fibre 
a union of regular curves with only normal crossings. 
Let $\XX$ be a regular proper model of $F$ 
with the reduced special fibre $X$ a union of regular 
curves with only normal crossings.

For a semi-global field $F$ and 
a regular proper model $\XX$ of $F$, one can 
associate three different kinds of overfields of $F$. 
We describe them below and discuss how they are 
related to each other. 

 For any point $x$ of $\XX,$ 
 let $R_x$ be the local ring at $x$ on $\XX,$ 
 $\widehat{R}_x$ the completion of the local ring $R_x,$ 
 $F_x$ the fraction field of $\widehat{R}_x$ and 
 $\kappa(x)$ the residue field at $x.$ 

For any subset $U$ of $X$ that is contained 
in an irreducible component of $X$, 
let $R_U$ be the subring of $F$ 
consisting of the rational functions which are 
regular at every point of $U.$ 
Let $\widehat{R_U}$ be the $t$-adic completion 
of $R_U$ and $F_U$ the fraction field of $\widehat{R_U}.$

Let $\eta \in X$ be a codimension zero point and $P \in X$ 
a closed point such that $P$ is in the closure $X_\eta$ of $\eta .$ 
Such a pair $(P,\eta )$ is called a {\it branch}. 
For a branch $(P,\eta ),$ we define $F_{P,\eta }$ 
to be the completion of $F_P$ at the discrete valuation 
of $F_P$ associated to $\eta .$ 
We call such fields {\it branch fields}. 
If $\eta$ is a codimension zero point of $X$, 
$U\subset X_\eta$ an open subset and 
$P \in X_\eta$ a closed point, then 
we will use $(P,U)$ to denote the 
branch $(P,\eta)$ and $F_{P,U}$ to denote 
the field $F_{P,\eta}.$

With $P, U, \eta $ as above, 
there are natural inclusions of 
$F_P,$ $F_U$ and $F_{\eta }$ 
into $F_{P,\eta }=F_{P,U}.$ 
Also, there is a natural inclusion 
of $F_{U}$ into $F_{\eta }.$

Let $\PP$ be a non-empty finite set of closed points of $X$ 
that contains all the closed points of $X$ 
where distinct irreducible components of $X$ meet. 
Let $\mathcal{U}$ be the set of connected components 
of the complement of $\PP$ in $X$ and let $\mathcal{B}$ 
be the set of branches $(P, U)$ with 
$P \in \PP$ and $U \in \mathcal{U}$ 
with $P$ in the closure of $U$.\\

Let $G$ be a linear algebraic group over $F.$ Let us define
 $$\displaystyle \Sh _{\XX, \PP}(F, G) := 
 \text{ker}\left( H^1(F,G) \rightarrow 
 \prod _{\xi \in \PP \cup \mathcal{U}} H^1 (F_{\xi }, G) \right).$$ 
 
 If $\XX$ is understood, we will just use the notation 
 $\Sh _{ \PP}(F, G).$\\
  
 Similarly, let us define $$\displaystyle \Sh _{\XX,X}(F,G) := 
 \text{ker}\left( H^1(F,G) \rightarrow 
 \prod _{P \in X} H^1(F_P,G)\right).$$ 

Again, if $\XX$ is understood, 
we will just use the notation $\Sh _{X}(F,G).$\\

 We have a bijection(\cite[Corollary 3.6.]{HHK1}):\hfill\\
 \begin{center}
 $\displaystyle \prod_{U\in \mathcal{U}} G(F_U) 
\mathlarger{ \mathlarger{ \mathlarger{\mathlarger{\backslash}}}} \prod_{(P,U) \in \mathcal{B}} G(F_{P,U}) 
\mathlarger{\mathlarger{\mathlarger{\mathlarger{/}}}} 
 \prod_{P\in \PP} G(F_P) \rightarrow \Sh _{\PP}(F, G) $
 \end{center}

This is a very useful result since 
the double coset mentioned above is usually more 
manageable to work with.

By (\cite[Corollary 5.9.]{HHK1}), we have 
$\displaystyle \Sh _{X}(F, G) = \bigcup \Sh _{\PP}(F, G)$ 
where union ranges over all finite subsets 
$\PP$ of closed points of $\XX$ 
which contain all the singular points of $X.$ 
We also have $\Sh _{X}(F, G) \subseteq \Sh (F, G)$ 
(\cite[Proposition 8.2.]{HHK1}).

\vspace{0.1in} 

\subsection{The associated graph}
 
We start with a basic fact about finite bipartite trees:

\begin{lemma}
 \label{factorisation_in_graphs}
  Let $\Gamma$ be a  finite bipartite graph and 
  $G$ be an abstract group. 
  Let $V$ be the set of vertices with parts 
  $V_1$ and $V_2.$ 
  For each edge $\theta$ of $\Gamma,$ 
  let $g_\theta \in G.$ 
  If $\Gamma$ is a tree, 
  then for every $v \in V,$ 
  there exists $g_v \in G$ such that if 
  $e$ is an edge joining two vertices $v_i \in V_i,$
  then $g_e = g_{v_1}g_{v_2}.$ 
 \end{lemma}
 
 \begin{proof}
  Suppose that $\Gamma$ is a tree. Without loss of generality, 
  we may assume that $\Gamma$ is a connected graph. 
  We prove the lemma by the induction on number of vertices. 
  Suppose that $\Gamma$ has one one vertex. 
  Then there is nothing to prove.

 Suppose that $\Gamma$ has more than one vertex. 
 Since $\Gamma$ is a connected tree, 
 there exists a vertex $v_0 \in V$ with exactly one edge 
 $\theta$ at $v_{0}.$ 
 Without loss of generality, we may assume $v_0 \in V_1.$ 
 Let $\Gamma'$ be the graph obtained from $\Gamma$ 
 by removing the vertex $v_0$ and the edge $\theta.$ 
 Then $\Gamma'$ is again a finite bipartite graph 
 which is a tree. 
 Then, by induction hypothesis, for every vertex $v$ of $\Gamma',$ 
 there exists $g_v \in G$ such that if $e$ is an edge in 
 $\Gamma'$ joining $v_1 \in V_1\setminus \{ v_0 \}$ and 
 $v_2 \in V_2,$ then $g_{e} = g_{v_1} g_{v_2}.$ 
 Let $v_0' \in V_2$ be the other vertex of the edge $\theta.$ 
 Let $g_{v_0} = g_{\theta} g_{v_0'}^{-1}.$ 
 Then it follows that $g_v$ have the required property. 
 \end{proof}

Let $T$ be a complete discretely valued ring with fraction field $K,$ 
and residue field $\kappa $. Let $t \in T$ be a parameter. 
Let $F$ be the function field of a smooth, projective, 
geometrically integral curve over $K$ 
and $\XX$ be a regular proper model of $F$ with 
the reduced special fibre $X$ a union of regular curves 
with only normal crossings. 
Let $\PP$ be a non-empty finite set of closed points of $X$ 
that contains all the closed points of $X$ 
where distinct irreducible components of $X$ meet. 
Let $\mathcal{U}$ be the set of connected components of 
the complement of $\PP$ in $X$ and let $\mathcal{B}$ be 
the set of branches $(P, U)$ with $P \in \PP$ 
and $U \in \mathcal{U}$ with $P$ in the closure of $U$.  

 We have a graph $\Gamma(\XX,\PP)$ associated to 
 $\XX$ and $\PP$ whose 
 vertices are elements of $\PP \cup \mathcal{U}$ 
 and edges are elements of $\mathcal{B}.$ 
Since there are no edges between any vertices 
which are in $\PP$ (resp. $\mathcal{U}$), 
$\Gamma(\XX, \PP)$ is a finite bipartite  graph with 
parts $\PP$ and $\mathcal{U}$. 
If $\PP'$ is another finite set of closed points of $X$ 
containing all the closed points of 
$X$ where  distinct irreducible components 
of $X$ meet, then $\Gamma(\XX,\PP)$ is a tree is and only if 
$\Gamma(\XX, \PP')$ is a tree (\cite[Remark 6.1(b)]{HHK1}). 
Hence if $\Gamma(\XX,\PP)$ is a tree for some $\PP$ as above, 
then we say that  the graph $\Gamma(\XX)$ 
associated to $\XX$ is a tree.\\ 

Now we have the following result as 
a corollary to Lemma \ref{factorisation_in_graphs}: 
\vspace{0.1in}  
 
 \begin{cor} 
 \label{tree}
 Let $F,$ $\XX,$ $X,$ $\PP,$ $\mathcal{U}$ and 
 $\mathcal{B}$ be as above. 
 Let $G$ be an abstract group and for each branch 
 $b \in \mathcal{B},$ let $g_b \in G.$ 
 Suppose that the graph $\Gamma(\XX)$ 
 associated to $\XX$ is a tree. 
 Then for every $\zeta \in \mathcal{U} \cup \PP,$ 
 there exists $g_\zeta \in G$ such that if 
 $b = (P, U) \in \mathcal{B},$ 
 then $g_b = g_P g_U.$
\end{cor}

\vspace{0.1in}

\subsection{ R-equivalence and R-trivial elements}\hfill \\
\label{Requivalence}
Other than patching, we use the notion of $R$-equivalence 
in this paper. We define the notion below and discuss results about 
$R$-equivalence for norm one tori, which will be used later.

\vspace{0.1in}

\begin{notation}
 Let $F$ be a field and $L$ be an \'etale algeba over $F.$ 
 Throughout this paper, we will denote the 
 norm $1$ torus $R^1_{L/F} \mathbb{G}_m$ by $T_{L/F}.$
\end{notation}

 Let $X$ be a variety over a field $F.$ 
 For a field extension $L$ of $F,$ 
 let $X(L)$ be the set of $L$-points of $X.$ 
 We say that two points $x_0, x_1 \in X(L)$ are 
 {\it elementary $R$-equivalent}, denoted by 
 $x_0 \sim x_1,$ if there is a rational map 
 $f: \mathbb{P}^1(L) \dashrightarrow X(L)$ 
 such that $f(0) = x_0$ and $f(1) = x_1.$ 
 The equivalence relation generated by $\sim$
  is called {\it $R$-equivalence}. 
 When $X = G$ is an algebraic group defined 
 over $F$ with the identity element $e,$ 
 we define $RG(L) = $ $\{$ $ x \in G(L):$ 
 $x$ is $R$-equivalent to $e \}.$ 
 The elements of $RG(L)$ are called 
 $R$-trivial elements. It is well-known 
 that $RG(L)$ is a normal subgroup of 
 $G(L)$ (cf. \cite[p-1]{G}) Sometimes, we denote 
 $G(L)/RG(L)$ by $G(L)/R$. 
 Let $L/F$ be a Galois extension 
 with Galois group $G,$ 
 and $T_{L/F}$ the norm $1$ 
 torus associated to the extension $L/F.$
 Then for any extension $N/F,$  $RT_{L/F}(N)$ 
 is the subgroup generated by the set
 $\{ a^{-1}\sigma (a): 
 a \in (L\otimes_F N)^{\times}, \sigma \in G \}$ 
(\cite[Proposition 15]{CTS}).\\ 

Now we discuss few basic results about 
$R$-equivalence on norm one tori:

\vspace{0.1in}

\begin{prop}
\label{product_of_field}
 Let $F$ be a field and $L_0/F$ a finite extension. 
 Let $L$ be the product of $r$ copies of $L_0.$ 
 Then the homomorphism $T_{L/F} \to T_{L_0/F}$ 
 given by $(a_1, \cdots , a_r) \mapsto a_1\cdots a_r$ 
 induces an isomorphism $T_{L/F}(F)/RT_{L/F}(F) \to 
 T_{L_0/F}(F)/RT_{L_0/F}(F).$ 
\end{prop}

\begin{proof}
 In fact the isomorphism 
 $(R_{L_0/F}(\mathbb{G}_m))^r \to 
 (R_{L_0/F}(\mathbb{G}_m))^r$ 
 given by sending $(b_1, \cdots , b_r)$ to 
 $(b_1, \cdots, b_{r-1}, b_1b_2\cdots b_r)$ 
 induces an isomorphism of algebraic groups 
 $\displaystyle T_{L/F} \to 
 (R_{L_0/F}(\mathbb{G}_m))^{r-1} \times 
 T_{L_0/F}$ (\cite[Lemma 1.1.]{FLP}). 
 Since $R_{L_0/F}(\mathbb{G}_m)$ is rational, it is
 $R$-trivial by (\cite[Corollary 1.6.]{G}).
 Hence the homomorphism 
 $T_{L/F} \to T_{L_0/F}$ 
 given by $(a_1, \cdots , a_r) 
 \mapsto a_1\cdots a_r$ 
 induces an isomorphism 
 $T_{L/F}(F)/RT_{L/F}(F) 
 \to T_{L_0/F}(F)/RT_{L_0/F}(F).$ 
\end{proof}

\vspace{0.1in}

\begin{cor}
\label{product_of_field_cor}
 Let $F$ be a field and $L_0/F$
  a finite extension of degree $d$ 
  and $L$ the product of $r$ copies of $L_0.$ 
  Suppose that $F$ contains $\rho,$ 
  a primitive $(dr)^{\rm th}$ root of unity. 
  If $T_{L_0/F}(F) = RT_{L_0/F}(F)<\rho^r>,$ 
  then $T_{L/F}(F) = RT_{L/F}(F)<\rho>.$ 
\end{cor}

\begin{proof}
 Since $(\rho, \rho, \cdots, \rho)$ 
 maps to $\rho^r$ under 
 the isomorphism given in (\ref{product_of_field}),
the corollary follows from (\ref{product_of_field}). 
\end{proof}

\vspace{0.1in}

\begin{lemma} 
\label{nthpower}
 Let $L/F$ be a finite Galois extension 
 of degree $m$ and $N/F$ any field extension. 
 If $\alpha \in (L \otimes_F N)^{\times},$ 
 then $N_{L \otimes_F N/N }(\alpha)^{-1} \alpha^m 
 \in RT_{L/F}(N).$
\end{lemma} 

\begin{proof}
 Let $G$ be the Galois group of $L/F.$ 
 Since $\displaystyle N_{L \otimes_F N/N}(\alpha) 
 = \prod_{\sigma \in G} \sigma(\alpha),$ 
 we have 
 $$\displaystyle N_{L \otimes_F N/N}(\alpha)^{-1} \alpha^m = 
 \frac{\alpha^m}{\prod_{\sigma \in G} \sigma(\alpha)} = 
 \prod_{\sigma \in G} \frac{\alpha}{\sigma(\alpha)}.$$ 
 Hence $N_{L \otimes_F N/N}(\alpha)^{-1} \alpha^m 
 \in RT_{L/F}(N).$
\end{proof}


 
\vspace{0.1in}

\section{Norm one elements - complete discretely valued fields}

Let $F$ be a complete discretely valued field with 
residue field $\kappa$. Let $L/F$ be a Galois extension. 
Let $l$ denote the residue field of $L.$ In this section, 
we investigate the relationship between the groups $T_{L/F}(F)/RT_{L/F}(F)$ 
and $T_{l/\kappa}(\kappa)/RT_{l/\kappa}(\kappa)$. 
This allows us to transfer our assumptions 
from the residue fields to the branch fields 
(see Subsection \ref{prelim_patching}), 
which is crucial for proving local-global principles 
for norm one tori for semi-global fields 
(see Lemma \ref{branch}).

\vspace{0.1in}

\begin{lemma}
\label{1modulo_maximal_ideal}
Let $F$ be a complete discretely valued field 
with residue field $\kappa$ and $L/F$ be 
a finite Galois extension of degree $n$ 
with residue field $l.$ 
Suppose that $n$ is coprime to char$(\kappa).$ 
Let $z \in T_{L/F}(F).$ 
If the image of $z$ in $l$ is 1, 
then $z \in RT_{L/F} (F).$ 
\end{lemma} 

\begin{proof}
 Let $S$ be the integral closure of $R$ in $L.$ 
 Then $S$ is a complete discrete valuation ring 
 with residue field $l.$ 
 Let $z \in T_{L/F}(F)$ with the 
 image of $z$ in $l$ is $1.$ 
 Since $n$ is coprime to char$(\kappa),$ 
 by Hensel's lemma, there is a $w \in S$ 
 with $\overline{w}=1$ and $z=w^n.$ 
 Since $N_{L/F}(z) = 1,$ $N_{L/F}(w)^n = 1$ 
 and hence $ \rho = N_{L/F}(w)$ is 
 an $n^{\rm th}$ root of unity. 
 Since $\overline{w} = 1,$ 
 $\overline{N_{L/F}(w)} = 
 N_{l/\kappa}(\overline{w})^e = 1,$ 
 where $e$ is the ramification index 
 of the extension $L/F.$ 
 Hence $\overline{\rho} = 1.$ 
 Since $n$ is coprime to char$(\kappa),$ 
 by Hensel's lemma, the quotient map $S \to l$ 
 induces a bijection from the set of $n^{\rm th}$ 
 roots of unity in $S$ to the set of $n^{\rm th}$ 
 roots of unity in $l.$ 
 Hence $\rho = 1$ and $w \in T_{L/F}(F).$ 
 Since $z = w^n,$ $z \in RT_{L/F}(F)$
  by (\ref{nthpower}). 
\end{proof}

\vspace{0.1in}

\begin{lemma} 
\label{residue_field_assumption}
Let $F$ be a complete discretely valued field with 
residue field $\kappa. $
Let $L/F$ be a Galois extension of degree $n.$ 
Suppose that $(n,char(\kappa ))=1.$ 
Suppose that  $F$ contains a primitive 
$n^{\rm th}$ root of unity $\rho _n$. 
Let $l$ be the residue field of $L$ 
and $f=[l:\kappa ].$ 
If $\displaystyle T_{l/\kappa }(\kappa ) = 
RT_{l/\kappa }(\kappa )<\rho_n^{n/f}>,$ 
then $\displaystyle T_{L/F}(F) = RT_{L/F}(F) <\rho_n>.$  
\end{lemma}

\begin{proof} 
Let $R$ be the discrete valuation ring of $F$ 
and $S$ be the integral closure of $R$ in $L.$ 
Let $e$ be the ramification index of the extension $L/F.$ 
Then $n = ef.$ For any element $y \in S($resp. $R),$ 
we will use $\overline{y}$ to denote 
its image in the residue field $l($resp. $\kappa ).$ 
 
 Let $x \in L$ with $N_{L/F}(x)=1.$ 
 Then, $\displaystyle N_{l/\kappa}(\overline{x})^e 
 = \overline{N_{L/F}(x)} = 1.$ 
 Hence $\displaystyle N_{l/\kappa }(\overline{x}) = \rho _n^{fi}$ 
 for some $i$ with $0 \leq i < e.$ 
 Let $y=\rho _n^{-i}x.$ Then $\displaystyle N_{L/F}(y)=1$ 
 and $\displaystyle N_{l/\kappa }(\overline{y})= 
 N_{l/\kappa }(\rho _n^{-i})N_{l/\kappa }(\overline{x})= 
 \rho _n^{-fi} \rho_n^{fi}=1.$ 
 Thus $\overline{y} \in T_{l/\kappa}(\kappa)$ 
 and hence, by the assumption, $\overline{y} = \theta \rho_n^{ej}$ 
 for some $\theta \in RT_{l/\kappa}(\kappa)$ and $j$ an integer. 
 Write 
$$\displaystyle \theta = \prod _{\sigma \in Gal(l/\kappa )}
(a_{\sigma })^{-1} \sigma (a_{\sigma })$$ 
for some $a_{\sigma } \in l^{\times}.$ 
Since $Gal(l/\kappa)$ is a quotient of $Gal(L/F),$ 
for every $\sigma \in Gal(l/\kappa )$ 
we choose a lift $\tilde{\sigma} \in Gal(L/F)$ of $\sigma.$ 
Let $b_\sigma \in S$ with $\overline{b}_\sigma = a_\sigma$ and 
$$\displaystyle z = y^{-1} \rho_n^{ej} \prod _{\sigma \in Gal(l/\kappa )}
(b_{\sigma })^{-1} \tilde{\sigma }(b_{\sigma }).$$ 
Then $z \in T_{L/F}(F)$ and $\overline{z} = 1.$ 
Thus, by (\ref{1modulo_maximal_ideal}), $z \in RT_{L/F}(F).$ 
Therefore $y \in RT_{L/F}(F) < \rho_n>$ and 
hence $x \in RT_{L/F}(F) < \rho_n>.$
\end{proof}

\vspace{0.1in}

\begin{defin}
A complete discretely valued field $K$ 
with finite residue field is called a $1$-{\it local} field. 
For $m \geq 1$, 
a complete discretely valued field $K$ with 
$m$-local 
residue field $k$  is called 
a $(m+1)$-{\it local} field.
If $K$ is a 1-local field, the residue field of $K$ is called the 
 {\it first  residue} field of $K$.
 If $K$ is a $(m+1)$-local field with residue field $k$, then the first residue field
 of $k$ is called the {\it first residue field} of $K$.
\end{defin}

 
\vspace{0.1in}

\begin{cor} 
\label{mlocal}
Let $K$ be an $m$-local field with first residue field 
$\kappa$ or an iterated Laurent series in $m$ variables
over an algebraically closed field $\kappa$. 
Let $L/K$ be a finite Galois extension of 
degree $n.$ If $n$ is coprime to char$(\kappa)$ and 
$K$ contains  a primitive $n^{\rm th}$ root of unity 
$\rho_n,$ then $T_{L/K}(K)= RT_{L/K}(K)<\rho_n>.$
\end{cor} 

\begin{proof}
 Every finite extension $l/\kappa$ is cyclic 
 and by Hilbert 90,  $T_{l/\kappa}(\kappa) = RT_{l/\kappa}(\kappa)$. 
 Thus, by (\ref{residue_field_assumption}), 
 $T_{L/K}(F)= RT_{L/K}(K)<\rho_n>.$ 
 The corollary follows by induction on $m$ and 
 by (\ref{residue_field_assumption}). 
 \end{proof}

\vspace{0.1in}

 \section{Two dimensional complete fields}

Let $F$ be a field with a discrete valuation $v$. 
 Let $\kappa(v)$ be the residue field of $v$. 
 Let $L/F$ be a finite extension and $w$ be a 
 discrete valuation on $L$ extending $v$.
 Let $e(w/v)$ be the ramification index of 
 $w$ over $v$. 
 For any field $E$, $a \in E^{\times}$ and $n \geq 1$, 
 let $E(\sqrt[n]{a})$ denote the field generated by 
 $E$ and $\sqrt[n]{a}$ in a fixed algebraic closure of $E$.\\

\vspace{0.1in} 

 \begin{lemma}
 \label{dvrextns}
 Let $F$ be a field with a discrete valuation $v$, 
 $\pi \in F^{\times}$ with $v(\pi) = 1$.
Let $L/F$ be a finite extension  of degree coprime to char$(\kappa(v))$ 
and $w$ be a discrete valuation of $L$ extending $v$. 
 Let $\ell$ be a prime not equal to char$(\kappa(v))$. 
 Then there is a unique discrete valuation $\tilde{v}$ on $F(\sqrt[\ell]{\pi})$ 
 extending $v$. Let $\tilde{w}$ on  $L(\sqrt[\ell]{\pi})$  be a discrete 
 valuation extending $w$. If  $\ell$ divides $e(w/v)$, 
 then $e(\tilde{w}/\tilde{v}) = e(w/v)/\ell$.
\end{lemma}

\begin{proof} 
Since $v(\pi) = 1$, $v$ is totally ramified in  $F(\sqrt[\ell]{\pi})$ 
and hence there is a unique extension 
$\tilde{v}$ of $v$ to $F(\sqrt[\ell]{\pi})$.

For the ramification index calculations, 
we can replace $F$ by $F_{v},$ the completion of 
$F$ with respect to the valuation $v$ and 
hence may assume that $F$ is complete. 
Let $L^{nr}$ be the maximal unramified 
subextension of $L/F$. 
Since the ramification index of $L/L^{nr}$ 
is same as the ramification index of $L/F$, replacing $F$ by $L^{nr}$, 
we may assume that $L/F$ is totally ramified. 
Since $n = e = [L : F]$ is coprime to char$(\kappa(v))$, 
we have $L = F(\sqrt[n]{u\pi})$ for some $u \in F$ with $v(u) =0$ 
(\cite[Proposition 1, p-32]{CFANT}). 

By hypothesis, we have that $\ell$ divides $n$.
Suppose that $u \in F^{\times \ell}$. 
Then $F(\sqrt[\ell]{\pi}) \subseteq L$ and 
hence $L/F(\sqrt[\ell]{\pi})$ is a 
totally ramified extension of degree $n/\ell$. 
Now suppose that $u \not \in F^{\times \ell}$. 
Then $L(\sqrt[\ell]{\pi}) = 
F(\sqrt[\ell]{\pi})(\sqrt[\ell]{u})(\sqrt[n/\ell]{\sqrt[\ell]{u\pi}})$. 
Hence the ramification index of the extension 
$L(\sqrt[\ell]{\pi})/F(\sqrt[\ell]{\pi})$ is $n/\ell$.
\end{proof}

\vspace{0.1in}

 \begin{notation} 
\label{notation_A_pi}
Let $A$ be a complete regular local  ring of dimension $2$ 
with residue field $\kappa$ and fraction field $F$. 
Let $\mathfrak{m} = (\pi_1,   \pi_2)  \subset A$ be the maximal ideal of $A$. 
Then, for $i=1, 2,$ we denote by $\widehat{A}_{(\pi_i)}$ be the completion 
of the local ring $A_{(\pi_i)}$ with respect to the ideal $(\pi_i)$ 
and by $ F_{\pi _i}$ the fraction field of $\widehat{A}_{(\pi_i)}.$
\end{notation}
 
We are studying these fields since 
the fields $F_P$ appearing in the 
patching setup (see Subsection \ref{prelim_patching}) are 
fraction fields of complete regular local rings 
of dimension $2$ and the branch fields 
are obtained as completions of the fields 
$F_P$ as discussed above. \\

 \vspace{0.1in}
 
\subsection{Structure of extensions of two dimensional complete fields} 

Let $F$ be as in Notation \ref{notation_A_pi}. 
In this subsection, we study Galois extensions of $F$.\\

In the following lemma, we prove that with some 
assumptions, a field extension $L$ over $F$ remains 
a field after base change to the completion $F_{\pi_i}$ 
for $i=1, 2:$

\vspace{0.1in}

\begin{lemma} 
\label{2dimextns}
Let $A$ be a complete regular local  ring of dimension $2$ 
with residue field $\kappa$ and fraction field $F$. 
Let $L/F$ be a field extension of degree $n$ 
where $n$ is coprime to char$(\kappa)$. 
Let $\mathfrak{m} = (\pi_1,   \pi_2)  \subset A$ be the maximal ideal of $A$. 
Suppose that $L/F$ is unramified on $A$ except possibly at 
$\pi_1$ and $\pi_2 \in A$. 
Then  $L\otimes_F F_{\pi_i}$ is a field for $i=1, 2$. 
\end{lemma} 

\begin{proof}  
Let $v_i$ be the discrete valuation of $F$ 
given by $\pi_i$, for $i = 1, 2$.
To show that $L\otimes_F F_{\pi_i}$ is a field, 
it is enough to show that there is 
a unique extension of $v_i$ to a discrete valuation 
on $L$. 

Let $w_i^j$ be the extensions of the valuations 
$v_i$ to $L$. 
Let $m$ be the maximum of $e(w_i^{j_i}/v_i)$ 
for $i = 1, 2;$ where $1 \leq j_i \leq n_i$ for some 
positive integers $n_1$ and $n_2$. 
Since each $e(w_i^{j_i}/v_i) \geq 1$, $m \geq 1$. 
We prove the result by induction on $m$.

 
Suppose that $m = 1$. 
Then $e(w_i^{j_i}/v_i)= 1$ for 
for $i = 1, 2$ and $1 \leq j_i \leq n_i$. 
Hence $L/F$ is unramified at $\pi_i$ for $i = 1, 2$. 
Since $L/F$ is unramified on $A$ except possibly at $\pi_1, \pi_2$, 
 $L/F$ is unramified on $A$. 
Let $\tilde{A}$ be the integral closure of $A$ in $L$. 
Then $\tilde{A}$ is again a complete regular local ring of dimension 
$2$ with $(\pi_1, \pi_2)$ as maximal ideal and fraction field $L$. 
Thus $\pi_i$ remains a prime over $\tilde{A}$. 
Hence there is a unique extension of 
$v_i$ to a discrete valuation of $L$. 
Hence $L\otimes_F F_{\pi_i} \cong L_{\pi_i}$ 
is a field.

Now suppose that $ m > 1$. 
Let $\ell$ be a prime which divides $m$. 
Let $E = F(\sqrt[\ell]{\pi_1}, \sqrt[\ell]{\pi_2})$ and 
$M = L(\sqrt[\ell]{\pi_1}, \sqrt[\ell]{\pi_2})$. 
Let $B$ be the integral closure of $A$ in $E$. 
Then, by (\cite[Corollary 3.3.]{PS}), $B$ is a regular local ring 
with maximal ideal 
$(\pi_1', \pi_2');$ where 
$\pi_1' = \sqrt[\ell]{\pi_1}$ and 
$\pi_2' = \sqrt[\ell]{\pi_2}$. 
Then  $M/E$ is unramified on $B$ except possibly 
 at $\pi_1'$ and $\pi_2'$. 
 Since $(\pi_1, \pi_2)$ is the maximal ideal of $A$, 
 it follows that  there is a unique extension of $v_i$ to $E$, 
 which we denote by $\tilde{v}_i$. 
 Let $\omega$ be a discrete valuation of $M$ 
 extending $\tilde{v}_i$ for some $i$. 
 Then the restriction of $\omega$ to $L$ is 
 equal to $w_i^{j_i}$ for some $j_i$.
 Let $E_i = F(\sqrt[\ell]{\pi_i})$ and 
 $M_i = L(\sqrt[\ell]{\pi_i})$.
 Let $\omega'$ and $v_i'$ be 
 the restrictions of $\omega$ and $\tilde{v}_i$
 to $M_i$ and $E_i$ respectively.  
 Suppose that $e(w_i^{j_i}/v_i) = m$. 
 Then, by (\ref{dvrextns}), 
 $e(\omega'/v'_i) =  e(w_i^{j_i}/v_i)/\ell$ and  
 $e(\omega/\tilde{v}_i) \leq e(\omega'/v'_i) =  
 e(w_i^{j_i}/v_i)/\ell$. 
  Hence,  by induction hypothesis, for each $i = 1, 2$, 
 there is a unique extension of $\tilde{v_i}$ to  $M$. 
 Since $L$ is a subfield of $M$, 
 there is a unique extension of $v_i$ to $L$.  
 Hence $L\otimes_F F_{\pi_i}$ is a field.
\end{proof}

The next lemma describes the structure of 
Galois extensions $L/F$ which are unramified on $A$ 
except possibly at $\pi_1$ and $\pi_2 \in A$ 
and totally ramified at $\pi_2$.

\vspace{0.1in}

\begin{lemma} 
\label{2dimGalextnsram}
Let $A$ be a complete regular local 
ring of dimension $2$ with residue field 
$\kappa$ and fraction field $F$. 
Let $L/F$ be a Galois  extension of degree 
$n$ where $n$ is coprime to char$(\kappa)$. 
Let $\mathfrak{m} = (\pi_1,   \pi_2)  \subset A$ 
be the maximal ideal of $A$. 
Suppose that $L/F$ is unramified on $A$ 
except possibly at $\pi_1$ and   $\pi_2  \in A$ 
and totally ramified at $\pi_2$.  
Then $L = F(\sqrt[n]{u\pi_1^m \pi_2})$ 
for some $u \in A$ a unit and 
some integer $m$. 
\end{lemma} 

\begin{proof} 
Let $G$ be the Galois group of $L/F$. 
Since the degree of $L/F$ is coprime to char$(\kappa)$ 
and $L/F$ is unramified on $A$ except possibly at $\pi_1$ and 
$\pi_2$, by (\ref{2dimextns}), $L\otimes_F F_{\pi_2}$ is a field. 
Since $L\otimes_F F_{\pi_2}/F_{\pi_2}$ is a totally tamely ramified extension, 
 $F_{\pi_2}$ contains a  primitive $n^{\rm th}$ root of unity and we have 
 $L\otimes_F F_{\pi_2}  = F_{\pi_2}(\sqrt[n]{\theta \pi_2})$
for some $\theta \in F_{\pi_2}$ which is a unit in the discrete valuation 
ring of $F_{\pi_2}$ by (\cite[Proposition 1, p-32]{CFANT}). 
In particular $G$ is a cyclic group. 
Since $F_{\pi_2}$ contains a   primitive $n^{\rm th}$ root of unity, 
the residue field $\kappa(\pi_2)$ of $F_{\pi_2}$ 
contains a  primitive $n^{\rm th}$ root of unity. 
Since $\kappa $   is the residue field of $\kappa(\pi_2)$, $\kappa$ 
also contains a  primitive $n^{\rm th}$ root of unity.
Since $A$ is complete, by Hensel's lemma, 
$F$ contains a primitive $n^{\rm th}$ root of unity.

Hence $L = F(\sqrt[n]{a})$ for some $a \in F$. 
Since $L/F$ is unramified on $A$ except possibly at $\pi_1, \pi_2$, 
we can choose $a = u \pi_1^m \pi_2^d$ for some $ u \in A$ a unit 
and integers $m, d$. 
Since $L/F$ is totally ramified at $\pi_2$, 
$d$ is coprime to $n$ and 
hence we can assume that $d = 1$. 
\end{proof}

Next we consider Galois extensions $L/F$ 
which are unramified on $A$ except possibly at $\pi_1$.

\vspace{0.1in}

\begin{lemma} 
\label{2dimGalextns1}
Let $A$ be a complete regular local  ring of dimension $2$ 
with residue field $\kappa$ and fraction field $F$. 
Let $L/F$ be a Galois  extension of degree coprime to char$(\kappa)$. 
Let $\mathfrak{m} = (\pi_1,   \pi_2)  \subset A$ be the maximal ideal of $A$. 
Suppose that $L/F$ is unramified on $A$ except possibly at $\pi_1$. 
Then there exists  a subextension $L_1/F$  of $L/F$ such that:\\

$\bullet$ $L_1/F$ is unramified on $A,$ and \\

$\bullet$ $L = L_1(\sqrt[e]{u\pi_1})$  for some unit $u$ 
in the integral closure of $A$ in $L_1$. \\
 
\end{lemma}

\begin{proof} 
Let $G$ be the Galois group of $L/F$. 
 Since $L\otimes_F F_{\pi_1}/F_{\pi_1}$ is a field extension 
 (\ref{2dimextns}), the Galois group 
 $Gal(L\otimes_F F_{\pi_1}/F_{\pi_1})$ 
 is isomorphic to $G$. 
 We will identify these two groups. 
 We consider the inertia group $H$ of the extension 
 $L\otimes_F F_{\pi_1}/F_{\pi_1}$ which is a subgroup of $G$.
 Let $L_1  = L^H$. Then $(L\otimes_F F_{\pi_1})^H = L_1 \otimes_F F_{\pi_1}$
  is unramified over $F_{\pi_1}$ by (\cite[Theorem 2, p-27]{CFANT}).
 Hence $({L_1})_{\pi_1} \cong L_1 \otimes_F F_{\pi_1}$ is unramified over
 $F_{\pi_1}$ and $L_1/F$ is unramified at $\pi_1$.
 Since $L/F$ is unramified on $A$ except possibly at $\pi_1$, 
 $L_1/F$ is unramified on $A$. 
 Then the integral closure  $B$  of $A$ in $L_1$ is a regular local ring with maximal ideal $(\pi_1, \pi_2)$. 
 Let $e = [L : L_1]$. 
 Since $L/F$ is unramified on $A$ except possibly at $\pi_1$, 
  $L/L_1$ is unramified on $B$ except possibly at 
 $\pi_1$. Hence by (\ref{2dimGalextnsram}), 
 with the roles of $\pi_1$ and $\pi_2$ interchanged, 
 we have $L = L_1(\sqrt[e]{u\pi_2^m \pi_1})$ for some $u \in B$ a unit. 
 Since $L/L_1$ is unramified on $B$ except possibly at $\pi_1$,  $m$ is divisibly by $e$ and hence 
 $L = L_1(\sqrt[e]{u \pi_1})$.
 \end{proof}

The next theorem describes the 
structure of Galois extensions $L/F$ 
unramified on $A$ except possibly 
at $\pi_1$ and $\pi_2 \in A$, hence generalising 
(\ref{2dimGalextnsram}) and (\ref{2dimGalextns1}).

\vspace{0.1in}

\begin{theorem} 
\label{2dimGalextns}
Let $A$ be a complete regular local ring of dimension $2$ 
with residue field $\kappa$ and fraction field $F$. 
Let $L/F$ be a Galois  extension of degree 
coprime to char$(\kappa)$. 
Let $\mathfrak{m} = (\pi_1,   \pi_2)  \subset A$ be 
the maximal ideal of $A$. 
Suppose that $L/F$ is unramified on 
$A$ except possibly at $\pi_1$ and $\pi_2 \in A$. 
Then there exists subfields $L_1$ and $L_2$ of $L$ such that:\\

$\bullet$ $F \subseteq L_1 \subseteq L_2 \subseteq L,$ \\

$\bullet$ $L_1/F$ is unramified on $A,$ \\

$\bullet$ $L_2 = L_1(\sqrt[d_1]{u\pi_1})$  for some unit $u$ in the integral closure of $A$ in $L_1,$ and\\

$\bullet$ $L = L_2(\sqrt[d_2]{v (\sqrt[d_1]{u\pi_1})^i \pi_2})$ for some unit $v$ in the integral closure of $A$ in $L_2$.
\end{theorem}

\begin{proof} 
Let $G$ be the Galois group of $L/F$.  
Since $L\otimes_F F_{\pi_2}/F_{\pi_2}$ is a field extension (\ref{2dimextns}), 
the Galois group $Gal(L\otimes_F F_{\pi_2}/F_{\pi_2})$ 
is isomorphic to $G$. 
We identify these two groups. 
We consider the inertia group $H$ of the extension 
$L\otimes_F F_{\pi_2}/F_{\pi_2}$ which is a subgroup of $G$. 
Let $L_2  = L^H$. 
Then, as in (\ref{2dimGalextns1}), 
$L_2/F$ is unramified on $A$ except possibly  at $\pi_1$.
Hence, by (\ref{2dimGalextns1}), 
there exists a sub extension $L_1/F$ of $L_2/F$ such that 
$L_1/F$ is unramified on $A$ and 
$L_2 = L_1(\sqrt[d_1]{u\pi_1})$  for some unit $u$ in 
the integral closure of $A$ in $L_1$.  
  Let $B$ be the integral closure of $A$ in $L_2$. 
  Then $B$ is a regular local ring with maximal 
  ideal  $(\sqrt[d_1]{u\pi_1}, \pi_2)$ by (\cite[Lemma 3.2.]{PS}). 
  Since $L/L_2$ is unramified on $B$ except possibly at 
  $\sqrt[d_1]{u\pi_1}$, $\pi_2$ and totally ramified at $\pi_2$, 
  by (\ref{2dimGalextnsram}), 
  $L = L_2(\sqrt[d_2]{v (\sqrt[d_1]{u\pi_1})^i \pi_2})$ for 
  some unit $v \in B$.
\end{proof}

 \vspace{0.1in}
 
\subsection{Norms over two dimensional complete fields}

Let $A$, $F$, $\pi_1, \pi_2$ 
be as in \ref{notation_A_pi}. 
Let $\lambda= u \pi _1^r \pi _2^s$ 
for some unit $u \in A$ and integers $r, s.$ 
In this subsection, we show that if 
$\lambda$ is a norm from the extension 
$L\otimes_F F_{\pi _1}/F_{\pi _1},$ then 
$\lambda$ is a norm from the extension $L/F.$\\

We begin by describing the elements of  
$F_{\pi_i}^{\times} / F_{\pi_i}^{\times n}:$

\begin{lemma}
\label{elements2dim}
Let $A$ be a complete regular local ring 
of dimension $2$ with residue field $\kappa$ 
and fraction field $F.$ 
Let $\mathfrak{m} = (\pi _1 , \pi _2 )$ 
be the maximal ideal of $A.$ 
Let $m \geq 1$ be an integer coprime to 
char$(\kappa ).$ 
Let $F_{\pi_1}$ be the completion of $F$ 
at the discrete valuation of $F$ given by $\pi_1$. 
Then every element in $F_{\pi_1}$ can be written as 
$u\pi_1^s\pi_2^t a^m$ for some  $u \in A$ a unit, 
$a \in F_{\pi_1}$ and integers $s, t$.
\end{lemma}

\begin{proof}
Let $\widehat{A}_{(\pi_1)}$ be the completion 
of the local ring $A_{(\pi_1)}.$ 
Then $ F_{\pi _1}$ is the fraction field of 
$\widehat{A}_{(\pi_1)}.$ 
Let $x \in  F_{\pi _1}.$ 
Then $x = v {\pi _1}^{s}$ for some unit 
$v \in \widehat{A}_{(\pi_1)}$ and integer $s.$ 
Let $\overline{v}$ be the image of $v$ in 
the residue field $\kappa(\pi_1)$ of $ F_{\pi_1}.$ 
Since $\kappa( \pi_1)$ is the fraction field of  $A/(\pi_1)$ 
and $A/(\pi_1)$ is a discrete valuation ring with the image 
$\overline{\pi}_2$ as a parameter, we can write 
$\overline{v} = z \overline{\pi _2}^{t}$ for 
some unit $z \in A/(\pi_1)$ and integer $t.$ 
Let $u \in A$ with 
$\overline{u} = z \in A/(\pi_1).$ 
Since $z$ is a unit in $A/(\pi_1),$ 
$u$ is a unit in $A.$ 
Hence $x^{-1} u \pi_1^{s} \pi_2^{t}$ 
maps to 1 in $\kappa( \pi_1).$ 
Since $m$ is coprime to char$(\kappa),$ 
by Hensel's lemma, 
$\displaystyle x = u {\pi _1}^{s} {\pi _2}^{t} {a}^{m}$ 
for some $a \in  F_{\pi _1 }.$ 
\end{proof}


\begin{lemma}
\label{2dim-norms-ram1} 
Let $A$ be a complete regular local ring of dimension $2$ with 
residue field $\kappa$ and fraction field $F.$ 
Let $L/F$ be a Galois field extension of degree $n$ 
where $n$ is coprime to char$(\kappa ).$ 
Let $\mathfrak{m} =(\pi _1 ,\pi _2)$ be the maximal ideal of $A.$ 
Suppose that $L/F$ is unramified on $A$ 
except possibly at $\pi _1$.
Let $\lambda = u \pi _1^r \pi _2^s$ for 
some unit $u \in A$ and integers $r, s.$ 
If $\lambda$ is a norm from the extension 
$L\otimes_F F_{\pi _1}/F_{\pi _1},$ then 
$\lambda$ is a norm from the extension $L/F.$ 
\end{lemma}

\begin{proof} 
Let $\mu \in L\otimes_F F_{\pi_1}$ be 
such that $N_{L\otimes_F F_{\pi_1}/F_{\pi_1}}(\mu) = \lambda$. 
 
 Since $L/F$ is a Galois extension which is unramified on $A$ 
except possibly at $\pi_1$, by (\ref{2dimGalextns1}), 
we have a subfield $F \subseteq L_1 \subseteq L$
such that $L_1/F$ is unramified on $A$ and 
$L = L_1(\sqrt[e]{v\pi_1})$ where $e=[L : L_1]$ 
and $v$ is a unit in the integral closure of $A$ in $L_1$. 
Let $B$ be the integral closure of $A$ in $L$. 
Then $B$ is a regular local ring with maximal ideal 
$(\sqrt[e]{v\pi_1}, \pi_2)$ by  (\cite[Lemma 3.2.]{PS}).
Hence, by (\ref{elements2dim}), 
$\mu = w \sqrt[e]{v\pi_1}^i \pi_2^j b ^n$ 
for some integers $i, j$, $b \in L\otimes_F F_{\pi_1}$
 and $w$ a unit in $B$. 
Let $\theta= w \sqrt[e]{v\pi_1}^i \pi_2^j  \in L$. 
Since $N_{L\otimes_F F_{\pi_1}/F_{\pi_1}}(\mu) = \lambda$,
we have $N_{L/F}(\theta^{-1}) \lambda  = 
N_{L\otimes_F F_{\pi_1}/F_{\pi_1}}(b^n)  \in F_{\pi_1}^{\times n}$. 
Since $N_{L/F}(\theta^{-1}) \lambda = 
[N_{L/F}(w) {v\pi_1}^{\frac{n}{e}i} \pi_2^{nj}]^{-1}
u \pi _1^r \pi _2^s
=[u (N_{L/F}(w))^{-1}] \pi_1^{r - \frac{n}{e}i} 
\pi_2^{s-nj}$ and   $u (N_{L/F}(w))^{-1}$ is a unit in $A$, 
  by (\cite[Corollary 5.5.]{PPS}), 
$N_{L/F}(\theta^{-1}) \lambda \in F^{\times n}$.
In particular $N_{L/F}(\theta^{-1})\lambda$ 
is a norm from the extension $L/F$
and hence $\lambda$ is a norm from $L/F$.
\end{proof}

\begin{theorem}
\label{2dim-norms} 
Let $A$ be a complete regular local ring of dimension $2$ with 
residue field $\kappa$ and fraction field $F.$ 
Let $L/F$ be a Galois field extension of degree $n$ where 
$n$ is coprime to char$(\kappa ).$ 
Let $\mathfrak{m} =(\pi _1 ,\pi _2)$ be the maximal ideal of $A.$ 
Suppose that $L/F$ is unramified on $A$ 
except possibly at $\pi _1 , \pi _2.$ 
Let $\lambda = u \pi _1^r \pi _2^s$ for 
some unit $u \in A$ and integers $r, s.$ 
If $\lambda$ is a norm from the extension 
$L\otimes_F F_{\pi _1}/F_{\pi _1},$ then 
$\lambda$ is a norm from the extension $L/F.$ 
\end{theorem}

\begin{proof} 

Let $\mu \in L\otimes_F F_{\pi_1}$ be such that 
$N_{L\otimes_F F_{\pi_1}/F_{\pi_1}}(\mu) = \lambda$. 
We show by induction on the degree of the field 
extension $L/F$ that $\lambda$ is a norm from the extension $L/F$. 

Since $L/F$ is a Galois extension which is unramified 
on $A$ except possibly at $\pi_1$and $\pi_2$, 
we have subfields $L_1$ and $L_2$ as in  (\ref{2dimGalextns}).

Let $B$ be the integral closure of $A$ in $L_2$. 
Then $B$ is a complete regular local ring with maximal ideal 
$(\sqrt[d_1]{v\pi_1}, \pi_2)$ by (\cite[Lemma 3.2.]{PS}). 
By (\ref{elements2dim}), we have 
$N_{L\otimes_F F_{\pi_1} / L_2 \otimes_F F_{\pi_1}}(\mu) = 
w \sqrt[d_1]{v\pi_1}^i \pi_2^j b ^n$ 
for some integers $i, j$, $b \in L_2\otimes_F F_{\pi_1}$ 
and $w$ a unit in $B$. 
Then $\theta = w \sqrt[d_1]{v\pi_1}^i \pi_2^j $ 
is a norm from $L\otimes_F F_{\pi_1}/
 L_2 \otimes_F F_{\pi_1}$.

Suppose that $F \neq L_2$. 
Then $[ L  : L_2] < [L : F]$ and  by induction, 
$\theta $ is a norm from $L/L_2$.
Write $\theta = N_{L/L_2}(\theta')$.
Then
$$
\begin{array}{rcl}
\lambda  & =  & N_{L\otimes_F F_{\pi_1}/F_{\pi_1}}(\mu)  \\
& = & N_{L_2\otimes_F F_{\pi_1}/F_{\pi_1}}(  N_{L\otimes_F F_{\pi_1}/ L_2 \otimes_ FF_{\pi_1}}(\mu)) \\
& = & N_{L_2\otimes_F F_{\pi_1}/F_{\pi_1}}(  \theta b^n)  \\
& = & N_{L_2\otimes_F F_{\pi_1}/F_{\pi_1}}(  \theta)  N_{L_2\otimes_F F_{\pi_1}/F_{\pi_1}}( b^n)  \\
& = & N_{L_2 /F }( \theta)   N_{L_2\otimes_F F_{\pi_1}/F_{\pi_1}}( b)^n  \\
& = & N_{L_2 /F }(N_{L/L_2}(\theta')   N_{L_2\otimes_F F_{\pi_1}/F_{\pi_1}}( b)^n  \\
& = & N_{L /F }( \theta')   N_{L_2\otimes_F F_{\pi_1}/F_{\pi_1}}( b)^n  \\
\end{array}
$$ 
Since $N_{L/F}(\theta') = N_{L_2/F}(\theta) = N_{L_2/F}(w \sqrt[d_1]{v\pi_1}^i \pi_1^j)$,
$N_{L/F}(\theta')^{-1} \lambda$ is a product of a unit in $A$ with a power of $\pi_1$ and a power of $\pi_2$.
Since $N_{L/F}(\theta')^{-1} \lambda =  N_{L_2\otimes_F F_{\pi_1}/F_{\pi_1}}( b)^n \in F_{\pi_1}^{n}$, 
  by (\cite[Corollary 5.5.]{PPS}),
 we conclude that $N_{L/F}(\theta')^{-1} \lambda  $
  is a $n^{\rm th}$ power in $F$ and hence a norm from $L$ to $F$.
  Hence $\lambda $ is also a norm from $L$ to $F$.

Now suppose $F = L_2$. 
Then $L=F(\sqrt[n]{v\pi_2})$ where 
$v$ is a unit in $A$ and hence $L/F$ is a cyclic extension of degree $n$. 
Let  $\sigma$ be a generator of the Galois group of $L/F$ and 
$C$ be the cyclic algebra $(L,  \sigma, \lambda)$.
Since $L/F$ is unramified on $A$ except at $\pi_2$, $C$ is a unramified on $A$ 
except possibly at  $\pi_1$ and $\pi_2$.
Since $\lambda$ is a norm from $L\otimes_F F_{\pi_1}$, 
$C\otimes_F F_{\pi_1}$ is a split algebra. 
 Thus, by (\cite[Corollary 5.5.]{PPS}),  
 $C$ is a split algebra and hence 
 $\lambda$ is a norm from the extension $L/F$ by 
 (\cite[Theorem 6, p-95]{A}). 
 \end{proof}

\vspace{0.1in}

\section{$\Sh $ vs $\Sh _{X}$}

In this section, we compare the groups 
$\Sh_X(F,T_{L/F})$ and $\Sh(F, T_{L/F})$ 
for a semi-global field $F$ 
and a finite Galois extension $L/F$ of degree coprime to the 
characteristic of the residue field. The proof uses 
Theorem \ref{2dim-norms} from the previous section. 
This allows us to use patching techniques to prove the 
local-global principle for norm one tori $T_{L/F}$ over $F$ 
with respect to discrete valuations.


\begin{theorem} 
\label{union_of_Sha_X_equals_Sha_dvr}
Let $K$ be a complete discretely valued field with 
residue field $\kappa.$ 
Let $F$ be the function field of a smooth, projective, 
geometrically integral curve over $K$ and 
$\XX_0$ a regular proper model of $F$ with reduced special fibre $X_0$.
Let $L/F$ be a Galois field extension of degree coprime to char $(\kappa ).$ 
 Then $\displaystyle \Sh (F,T_{L/F}) = \bigcup_{X} \Sh _X (F,T_{L/F}),$ 
where $X$ is  running over the reduced special fibres 
of regular proper models $\XX$ of $F$ which are 
obtained as a sequence of blow-ups of 
$\XX_0$ centered at closed points of $\XX_0$.
\end{theorem}

\begin{proof} 
Let $x \in \Sh (F,T_{L/F}) \subseteq H^1(F, T_{L/F}).$ 
Since $H^1(F, T_{L/F}) \simeq F^{\times}/N_{L/F}(L^{\times}),$ 
let $\lambda \in F^{\times} $ be a lift of $x.$ 
For a regular proper model $\XX$ of $F,$ 
let $supp_{\XX}(\lambda)$ denote the support of 
$\lambda$ in $\XX$ and $ram_{\XX}(L/F)$ denote 
the ramification locus for the extension $L/F$ with respect 
to $\XX$. 
By (\cite[p-193]{L}), there exists a  sequence of blow-ups $\XX \to \XX_0$ 
centered at closed points of $\XX_0$
 such that the union of $supp_{\XX}(\lambda ),$ $ram_{\XX}(L/F)$ 
and the reduced special fibre $X$ of $\XX$ 
is a union of regular curves with normal crossings. 
We show that $x \in \Sh _{X}(F,T_{L/F}).$ 

 Let $P\in X.$ 
 First suppose $P$ is a generic point of $X.$ 
 Then $P$ gives a discrete valuation 
 $\nu$ of $F$ with $F_\nu = F_P.$ 
 Since $x \in \Sh (F,T_{L/F}),$ $x$ maps to 
 $0$ in $H^1(F_P, T_{L/F}).$

 Next suppose that $P$ is a closed point. 
 Let $\eta_1$ be the generic point of 
 an irreducible component of $X$ containing $P.$ 
 Let $\mathcal{O}_{\XX,P}$ be the local ring at $P$ 
 and $\mathfrak{m}_{\XX,P}$ be its maximal ideal. 
 Then, by our choice of $\XX,$ 
 $\mathfrak{m}_{\XX,P}= (\pi _1, \pi _2)$ where 
 $\pi _1$ is a prime defining $\eta_1$ at $P,$ 
 $\lambda = u \pi _1^r \pi _2^s$ for some unit 
 $u\in \mathcal{O}_{\XX,P}$ and integers $r,s,$ and $L\otimes_F F_{P}/F_{P}$ is 
 unramified on $\mathcal{O}_{\XX,P}$ except possibly at $\pi _1, \pi _2.$ 
 Since $L/F$ is a Galois extension, $L\otimes_F F_{P} = \prod L_P$ 
 for some Galois extension $L_P/F_P.$ 
 Since $L\otimes_F F_{P}/F_P$ is unramified on 
 $\mathcal{O}_{\XX,P}$ except possibly at $\pi _1, \pi _2,$ 
 $L_P/F_P$ is unramified on $\mathcal{O}_{\XX,P}$ 
 except possibly at $\pi _1, \pi _2.$ 
 Since $\lambda $ is a lift of $x \in \Sh _{X}(F,T_{L/F}),$ 
 $\lambda $ is a norm from $L\otimes_F F_{\eta_1}/F_{\eta_1}.$ 
 Since $F_{\eta_1} \subset F_{P, \eta_1},$ 
 $\lambda$ is a norm from $L\otimes_F F_{P, \eta_1}/F_{P, \eta_1 }.$ 
 Hence $\lambda $ is a norm from 
 $L_P \otimes_{F_P} F_{P, \eta_1} /F_{P, \eta_1}.$ 
 Thus, by (\ref{2dim-norms}), 
 $\lambda$ is a norm from $L_P/F_P$ and 
 $x$ maps to $0$ in $H^1(F_P, T_{L/F}).$ 
 Therefore $x \in \Sh _X (F,T_{L/F}).$ 
 By (\cite[Proposition 8.2.]{HHK1}), 
 we have $\displaystyle \bigcup _{X} \Sh _X (F,T_{L/F})= \Sh (F,T_{L/F}),$ 
 where $X$ is  running over the reduced special fibres 
of regular proper models $\XX$ of $F$ which are 
obtained as a sequence of blow-ups of 
$\XX_0$ centered at closed points of $\XX_0$.
\end{proof}

\begin{remark}
\label{divisorial}
The proof of (\ref{union_of_Sha_X_equals_Sha_dvr}) 
also works if we just consider divisorial discrete valuations 
instead of considering all discrete valuations on $F$.
\end{remark}

\vspace{0.1in}

\section{Local Global Principle}

In this section, we prove the local-global 
principle for norm one tori over semi-global fields with respect to 
points on the reduced special fibre under some assumptions on 
the semi-global field and the residue field. This, combined with 
Theorem \ref{union_of_Sha_X_equals_Sha_dvr}, allows us to 
conclude that under the same assumptions, the local-global principle for 
norm one tori also holds with respect to discrete valuations.

\begin{lemma}
\label{branch} 
Let $K$ be a complete discretely valued field with 
residue field $\kappa $ and 
$F$ be 
the function field of a smooth, projective, 
geometrically integral curve over $K.$ 
Let $\XX$ be a regular proper model of $F$ with 
the reduced special fibre $X$ 
a union of regular curves with normal crossings. 
Let $L/F$ be a Galois extension over $F$ of degree $n.$ 
Let $P \in X$ be a closed point and 
$U$ an irreducible open subset of 
$X$ with $P$ in the closure of $U.$ 
Suppose that: \\

 $\bullet$ $n$ is coprime to char$(\kappa ),$ \\
 
 $\bullet$ $K$ contains a primitive $n^{\rm th}$ root of unity $\rho,$ and\\
 
 $\bullet$ for all finite Galois extensions 
 $l/\kappa(P)$ of degree $d$ dividing $n,$ 
 $$\displaystyle T_{l/\kappa(P)}(\kappa(P)) = 
 RT_{l/\kappa(P)}(\kappa(P)) <\rho^{\frac{n}{d}}>.$$ 
 
 Then $\displaystyle T_{L\otimes_F F_{P,U}/F_{P, U}}(F_{P,U}) = 
 RT_{L\otimes_F F_{P, U}/F_{P,U}}(F_{P,U})<\rho>.$ 
\end{lemma}

\begin{proof}
 Let $\kappa(U)$ be the function field of $U.$ 
 Since $X$ is a union of regular curves, 
 $P$ gives a discrete valuation on $\kappa(U).$ 
 Let $\kappa(U)_P$ be the completion of $\kappa(U)$ at $P.$ 
 Then, by definition, $F_{P, U}$ is a complete discretely valued field 
 with residue field $\kappa(U)_P.$ 
 Since $L/F$ is a Galois extension of degree $n,$ 
 $L \otimes_F F_{P, U} \simeq \prod L_0$ for 
 some finite Galois extension $L_0/F_{P, U}$ of 
 degree $d$ dividing $n.$ 
 Since $F_{P, U}$ is a complete discretely valued field 
 with residue field $\kappa(U)_P,$ $L_0$ is a 
 complete discretely valued field with residue field $M_0$ 
 a finite extension of $\kappa(U)_P$ 
 of degree $d_1$ dividing $d.$ 
 Since $\kappa(U)_P$ is a complete discretely valued field with 
 residue field $\kappa(P),$ 
 $M_0$ is a complete discretely valued field with 
 residue field $l_0$ a finite Galois extension of $\kappa(P)$ 
 of degree $d_2$ dividing $d_1.$
Hence, by the assumption on $\kappa(P)$ and 
(\ref{residue_field_assumption}), we have 
$$T_{M_0/\kappa(U)_P}(\kappa(U)_P) = 
RT_{M_0/\kappa(U)_P}(\kappa(U)_P)
<\rho^{\frac{n}{d_1}}>.$$ 
Hence, once again by (\ref{residue_field_assumption}), 
we have $$T_{L_0/F_{P, U}}(F_{P, U}) = 
RT_{L_0/F_{P, U}}(F_{P, U})<\rho^{n/d}>.$$ 
Since $L\otimes_F F_{P, U}$ is the product of $\frac{n}{d}$ 
copies of $L_0,$ by (\ref{product_of_field_cor}), 
we have $$T_{L\otimes_F F_{P, U}/F_{P, U}}(F_{P, U}) = 
RT_{L\otimes_F F_{P, U}/F_{P, U}}(F_{P, U}) <\rho>.$$ 
\end{proof}

We are now ready to prove the local-global 
principle for norm one tori over semi-global fields 
with respect to points on the reduced special fibre of 
the model under some assumptions on the given semi-global 
field $F$ and the residue field $\kappa$.

\begin{theorem} 
\label{Sha_X_vanishes}
 Let $K$ be a complete discretely valued field 
 with residue field $\kappa $ and 
 $F$ be 
 the function field of a smooth, projective, 
geometrically integral curve over $K.$ 
 Let $\XX$ be a regular proper model of $F$ with 
 reduced special fibre $X$ 
 a union of regular curves with normal crossings. 
 Let $L/F$ be a Galois extension over $F$ of degree $n.$ 
 Suppose \\
 
 $\bullet$ $n$ is coprime to char$(\kappa ),$ \\
 
 $\bullet$ $K$ contains a primitive $n^{\rm th}$ root of unity $\rho,$ \\
 
 $\bullet$ for all finite extensions $\kappa'/\kappa $ and 
 for all finite Galois extensions $l/\kappa'$ of degree $d$ dividing $n,$ 
$$T_{l/\kappa '}(\kappa ') = RT_{l/\kappa'}(\kappa') <\rho^{\frac{n}{d}}>,$$ \\
$\bullet$ the graph associated to $\XX$ is a tree.\\

Then $ \Sh _{X}(F, T_{L/F}) = 0.$ 
\end{theorem}

\begin{proof}
 Let $\mathcal{P}$ be a finite set of closed points of $X$ 
 containing all the nodal points of $X.$ 
By (\cite[Corollary 5.9.]{HHK1}), it is enough to show that 
$\Sh _{\mathcal{P}}(F, T_{L/F})=0.$ 
Let $X \setminus \mathcal{P} = \cup_i U_i.$ 
Then each $U_i$ is an irreducible open subset of $X.$ 
By (\cite[Corollary 3.6.]{HHK1}), it is enough to show that 
the product map 
$$ \displaystyle \psi: \prod_i T_{L/F} (F_{U_i}) \times 
\prod_{P \in \mathcal{P}} T_{L/F}(F_P) \to 
\prod _{(P,U_i)} T_{L/F}(F_{P,U_i})$$ 
is onto, 
where the product on the right hand side is taken 
over all pairs $(P, U_i)$ with $P \in \mathcal{P}$ and 
$U_i$ such that $P$ in the closure of $U_i.$ 

 Let $\displaystyle (\lambda _{P,U_i}) 
 \in \prod _{(P,U_i)} T_{L/F}(F_{P,U_i}).$ 
 We show that $\displaystyle (\lambda _{P,U_i})$ 
 is in the image of $\psi.$ 
 By (\ref{branch}), for each pair $(P, U_i)$ with 
 $P$ in the closure of $U_i,$ we have 
 $ \displaystyle \lambda _{P,U_i} = 
 \rho^{j_{P,U _i}} \mu _{P,U_i}$ 
 for some integer $j_{P,U_i}$ and 
 $\mu _{P,U_i} \in RT_{L/F}(F_{P,U_i}).$ 
 Let $G$ be the Galois group of $L/F.$ 
 For each $\sigma \in G,$ there exists 
 $a_{\sigma, P, U_i} \in (L\otimes_F F_{P, U_i})^{\times}$ 
 such that $$\displaystyle \mu_{P,U_i} = 
 \prod _{\sigma \in G(L/F)} \sigma (a_{\sigma ,P,U_i}) 
 (a_{\sigma ,P,U_i})^{-1}.$$ 
 
 Since the group $R_{L/F}(\mathbb{G}_m)$ is $F$-rational, 
 by (\cite[Theorem 3.6.]{HHK}), 
 $$ \displaystyle \prod_i (L\otimes_F F_{U_i})^{\times} \times 
 \prod_{P \in \mathcal{P}} (L\otimes_F F_P)^{\times} \to 
 \prod _{(P,U_i)} (L\otimes_F F_{P,U_i})^{\times}$$ is onto. 
 Hence for each $\sigma \in G,$ there exist 
 $b_{\sigma, U_i} \in (L\otimes_F F_{U_i})^{\times}$ and 
 $b_{\sigma, P} \in (L\otimes_F F_P)^{\times}$ such that 
 $a_{\sigma, P, U_i} = b_{\sigma, U_i} b_{\sigma, P}.$ 
We have 
$$
\begin{array}{rcl}
 \mu_{P,U_i} & = & \displaystyle \prod _{\sigma \in G(L/F)} \sigma (a_{\sigma ,P,U_i}) (a_{\sigma ,P,U_i})^{-1} \\
 & = & \displaystyle \prod _{\sigma \in G(L/F)} \sigma (b_{\sigma, U_i} b_{\sigma, P}) (b_{\sigma, U_i} b_{\sigma, P})^{-1} \\
 & = & \displaystyle \prod _{\sigma \in G(L/F)} \sigma (b_{\sigma, U_i}) (b_{\sigma, U_i} )^{-1} \sigma(b_{\sigma, P}) ( b_{\sigma, P})^{-1}.
 \end{array}
 $$ 
 Since $\sigma (b_{\sigma, U_i}) (b_{\sigma, U_i} )^{-1} 
 \in T_{L/F}(F_{U_i})$ and 
$\sigma(b_{\sigma, P}) ( b_{\sigma, P})^{-1} 
\in T_{L/F}(F_P),$ $(\mu_{P, U_i})$ 
is in the image of $\psi.$

 Since $T_{L/F}(F)$ is a group and $\rho \in T_{L/F}(F),$ 
 by (\ref{tree}), $(\rho^{j_{P,U_i}})$ is in the image of $\psi.$ 
 Since $\psi$ is a homomorphism, $\displaystyle (\lambda _{P,U_i})$ 
 is in the image of $\psi,$ hence proving that $\psi$ is onto.
\end{proof}

Using the above theorem and 
Theorem \ref{union_of_Sha_X_equals_Sha_dvr}, 
we get that the under same assumptions, 
we have the local-global principle for 
norm one tori over semi-global fields 
with respect to discrete valuations:

\begin{theorem} 
\label{dvrsha}
With the hypothesis as in Theorem \ref{Sha_X_vanishes}, 
we have $\Sh (F, T_{L/F})=0.$
\end{theorem}

\begin{proof}
Let $\XX$ be a regular proper model of $F$ 
which is obtained as a sequence of blow-ups of 
$\XX_0$ at closed points. Since the graph 
$\Gamma(\XX_0)$ is a tree, 
$\Gamma(\XX)$ is also a tree (\cite[Remark 6.1(b)]{HHK1}). 
Let $X$ be the reduced special fibre of $\XX$. 
Then $\Sh _{X}(F, T_{L/F}) = 0$  (\ref{Sha_X_vanishes}). 
Thus, by (\ref{union_of_Sha_X_equals_Sha_dvr}), 
we have $\Sh(F, T_{L/F}) = 0.$ 
\end{proof}

 
\begin{cor}
\label{sha_m_local}
 Let $K$ be an $m$-local field with first residue field $\kappa$ 
or an iterated Laurent series in $m$ variables
over an algebraically closed field $\kappa$. 
Let $F$ be the function field of a smooth, projective, geometrically integral curve over $K$ and 
$L/F$ be a finite Galois extension of degree $n$ with $(n,char(\kappa ))=1.$ 
 Let $\XX$ be a regular proper model of $F$ with reduced special fibre 
 $X$ a union of regular curves with normal crossings. 
 Suppose that the graph associated to $\XX$ is a tree. 
 If $K$ contains a primitive $n^{\rm th}$ root of unity, 
 then $\Sh (F, T_{L/F})=0.$ 
\end{cor}

\begin{proof}
 Follows from (\ref{mlocal}) and (\ref{dvrsha}).
\end{proof}

\begin{remark}
By (\ref{divisorial}), the results also hold true if 
we just consider divisorial discrete valuations 
instead of all discrete valuations on $F$ 
in (\ref{dvrsha}) and (\ref{sha_m_local}).

\end{remark}

\vspace{0.1in}

\section{Counterexamples}

Let $K$ be a complete discretely valued field 
with residue field algebraically 
closed. Colliot-Th\'el\`ene, Parimala and Suresh 
(\cite[Section 3.1. \& Proposition 5.9.]{CTPS1}) 
constructed a function field
of a smooth, projective, 
geometrically integral curve over $K$ and a 
Galois extension $L/F$ with 
Galois group $\Z/2\Z \times \Z/2\Z$ 
such that the local-global principle fails for 
the norm one torus $T_{L/F}$ associated to $L/F$.
They use higher reciprocity laws to detect 
non-trivial elements in $\Sh(F, T_{L/F})$.
In this section, we produce examples of 
Galois extensions $L/F$ with Galois group 
$\Z/n\Z \times \Z/n\Z$ 
and using patching techniques, 
we show that $\Sh(F, T_{L/F})$ 
is non-trivial. 

Let $k$ be a number field and $L_1, L_2$ be two Galois extensions of $k$.
Let $T$ be the $k$-torus given by $N_{L_1/k}(z_1) N_{L_2/k}(z_2) = 1$.
If $L_1$ and $L_2$ are linearly disjoint, then Demarche and Wei 
(\cite[Theorem 1]{DW}) proved that 
the local-global principle holds for $T$. 
In this section, we also give an example to 
show that a similar result does not hold in general 
for function fields of curves over a complete discretely valued field.

\begin{prop}
\label{ufd}
 Let $A$ be a unique factorization domain 
 and $F$ be its fraction field. 
 Let $L/F$ be a finite Galois extension 
 and $B$ be the integral closure of $A$ in $L.$ 
 Suppose that $B$ is a unique factorization domain. 
 Then every element in $T_{L/F}(F)$ can be written as 
 $s \theta $ for some $s \in RT_{L/F}(L)$ and 
 $\theta \in B$ a unit. 
\end{prop}

\begin{proof}
 Let $\lambda \in T_{L/F}(F).$ 
 Then $\lambda \in L^{\times}$ and $N_{L/F}(\lambda) = 1.$ 
 Since $L$ is the fraction field of $B,$ 
 $\displaystyle \lambda = \frac{\alpha}{\beta}$ 
 for some $\alpha, \beta \in B.$ 
 Since $N_{L/F}(\lambda) = 1,$ $N_{L/F}(\alpha) = N_{L/F}(\beta).$ 
 Let $p \in B$ be a prime. 
 Since $A$ is a unique factorization domain, 
 $pB \cap A = qA$ for some prime $q \in A$ and 
 $N_{L/F}(p) = v q^r$ for some unit $v \in A.$ 
 Suppose that $p $ divides $\alpha$ in $B.$ 
 Then $N_{L/F}(p)$ divides $N_{L/F}(\alpha)$ in $A$ 
 and hence $q$ divides $N_{L/F}(\alpha).$ 
 Since $N_{L/F}(\alpha) = N_{L/F}(\beta),$ 
 there exists a prime $p' \in B$ such that 
 $p'$ divides $\beta$ and $p'B \cap A = qA$. 
 Since $L/F$ is a Galois extension, 
 there exists $\sigma \in Gal(L/F)$ such that 
 $p = w\sigma(p')$ for some unit $w \in B$. 
 Write $\alpha = p \alpha'$ and $\beta = p' \beta'.$ 
 Then $\displaystyle \lambda = \frac{\alpha}{\beta} = 
 \frac{p}{p'} \frac{\alpha'}{\beta'} = 
 \frac{\sigma{p'}}{p'}\frac{w\alpha'}{\beta'}.$ 
 Since $B$ is a unique factorization domain, 
 the proposition follows by induction on the 
 number of prime factors of $\alpha $ in $B.$
\end{proof}

\begin{prop} 
\label{R-trivial_times_rho}
 Let $A$ be a complete regular local ring of dimension $2$ with 
 maximal ideal $(\pi, \delta),$ fraction field $F$ and residue field $\kappa.$ 
 Let $n$ be a positive integer which is coprime to char$(\kappa).$ 
 Let $L = F(\sqrt[n]{\pi}, \sqrt[n]{\delta}).$ 
 Suppose that $F$ contains a 
 primitive $n^2$- th root of unity $\rho.$ 
 Then $T_{L/F}(F) = RT_{L/F}(F)<\rho>.$ 
\end{prop}
 
\begin{proof}
 Let $B$ be the integral closure of $A$ in $L.$ 
 Then $B$ is a regular local ring of dimension $2$ with
 fraction field $L$ and residue field $\kappa$ (\cite[Corollary 3.3.]{PS}). 
 Let $\lambda \in T_{L/F}(F).$ 
 Then $\lambda \in L^{\times}$ with $N_{L/F}(\lambda) = 1.$ 
 Then, by (\ref{ufd}), there exists $s \in RT_{L/F}(F)$ 
 and a unit $\theta \in B$ such that $\lambda = s \theta.$ 
 Since the residue fields of $A$ and $B$ are equal, 
 there exists $\theta_1 \in A$ such that $\theta \equiv \theta_1$ 
 modulo the maximal ideal of $B.$ 
Since $n$ is coprime to char$(\kappa),$ 
by Hensel's lemma, we have $\theta = \theta_1\alpha^{n^2}$ 
for some unit $\alpha \in B.$ 
Let $s_1 = N_{L/F}(\alpha)^{-1}\alpha^{n^2} \in L.$ 
Then, by (\ref{nthpower}), $s_1 \in RT_{L/F}(F).$ 
Let $a = \theta_1N_{L/F}(\alpha) \in F.$ 
Then $\theta = a s_1.$ Thus 
$\lambda = s \theta = s a s_1= s s_1 a.$ 
Since $N_{L/F}(\lambda )= 1= N_{L/F}(s) = 
N_{L/F}(s_1),$ $1= N_{L/F}(a) = a^{n^2}.$ 
Thus $a \in <\rho>.$ 
Hence $\lambda = s s_1 a \in RT_{L/F}(F)<\rho>.$ 
\end{proof}

\begin{lemma}
\label{rho_non-trivial}
 Let $F$ be a complete discretely valued field
 with residue field $\kappa$ and ring of integers $R.$ 
 Let $n$ be a positive integer coprime to char$(\kappa).$ 
 Let $\pi \in R$ be a parameter and $u \in R$ a unit 
 with $\displaystyle [F(\sqrt[n]{u}):F] = n.$ 
 Let $\displaystyle L = F(\sqrt[n]{u}, \sqrt[n]{\pi}).$ 
 Suppose that $F$ contains a 
 primitive $n^2$-th root of unity $\rho.$ 
 Then $\displaystyle \rho^t \in RT_{L/F}(F)$ 
 if and only if $n$ divides $t.$ 
\end{lemma}

\begin{proof}
 Let $\sigma$ be the automorphism of $L/F$ given by 
 $\displaystyle \sigma(\sqrt[n]{\pi}) = 
 \rho ^n \sqrt[n]{\pi}$ and $\sigma(\sqrt[n]{u}) = 
 \sqrt[n]{u}$ and 
 $\tau $ be the automorphism $L/F$ given by 
 $\displaystyle \tau(\sqrt[n]{u}) = \rho ^n \sqrt[n]{u}$ 
 and $\tau(\sqrt[n]{\pi}) = \sqrt[n]{\pi}.$ 
 Then the Galois group of $L/F$ is an abelian group of order 
 $n^2$ generated by $\sigma$ and $\tau$ and  hence $\displaystyle RT_{L/F}(F)$ 
 is generated by the set 
 $\{ \frac{\sigma(a)}{a} \frac{\tau(b)}{ b} \mid a, b \in L^{\times} \}.$ 
 Since $\displaystyle \rho^n = \tau(\sqrt[n]{u})/ \sqrt[n]{u} \in RT_{L/F}(F),$ 
 $\rho^{nj} \in RT_{L/F}(F)$ for any integer $j.$

Conversely, suppose $\rho^t \in RT_{L/F}(F)$ for some integer $t.$ 
Without loss of generality, we may assume that $1\leq t \leq n^2.$ 
Then $\displaystyle \rho^t = a^{-1}\sigma (a) b^{-1}\tau (b)$ 
for some $a, b \in L.$ 
Let $L' = F(\sqrt[n]{\pi}).$ Since 
$\rho \in F$ and $N_{L/L'} (b^{-1}\tau(b)) = 1,$ 
we have 
$\displaystyle \rho^{nt} = N_{L/L'}(a)^{-1} N_{L/L'}(\sigma(a)).$ 
Let $c = N_{L/L'}(a) \in L'.$ 
Since $\displaystyle \sigma(c) = 
\sigma(N_{L/L'}(a)) = N_{L/L'}(\sigma(a)),$ 
we have $\sigma(c) = \rho^{nt} c.$ 
Hence $\displaystyle \sigma(c^n) = 
(\sigma(c))^n = (\rho^{nt})^n c^n = c^n.$ 
Since $L'/F$ is a Galois extension with Galois group generated 
by $\sigma,$ $c^n \in F.$ 
Thus $\displaystyle c = \theta \sqrt[n]{\pi}^m$ for 
some integer $m$ and $\theta \in F$. 
Since $L/L'$ is an unramified extension of degree $n$ 
and $c$ is a norm from $L/L',$ 
the valuation of $c$ is divisible by $n$. 
Since $\theta \in F$ and $\sqrt[n]{\pi}$ is a parameter in $L',$ 
$m = nr$ for some $r.$ Hence $c \in F$ and 
$\rho^{nt} = c^{-1}\sigma(c) = 1.$ 
Since $\rho$ is a primitive $n^2$-th root of unity, 
$n$ divides $t.$ 
 \end{proof}

\begin{notation}
\label{three_maximal_ideals}
 Let $A$ be a semi-local regular ring of dimension $2$ 
 with three maximal ideals $m_1, m_2, m_3.$ 
 Suppose that there exist three prime elements 
 $\pi_1, \pi_2, \pi_3 \in A$ such that 
 $m_1 = (\pi_2, \pi_3),$ $m_2= (\pi_1, \pi_3)$ and 
 $m_3 = (\pi_1, \pi_2).$ 
 Suppose that $\pi_i \notin m_i$ for all $i.$ 
 Let $n \geq 2$ be an integer coprime to char$(A/m_i)$ for all $i.$ 
 Let $F$ be the fraction field of $A.$ 
 For $1 \leq i \leq 3,$ let $\widehat{A}_{m_i}$ 
 be the completion of $A$ at $m_i,$ 
 $F_{m_i}$ be the fraction field $\widehat{A}_{m_i}$ 
 and $F_{\pi_j}$ be the completion of $F$ at the 
 discrete valuation given by $\pi_j.$ 
 Let $1 \leq i \neq j, k\leq 3.$ 
 Since $m_i = (\pi_j, \pi_k),$ $\widehat{A}_{m_i}$ 
 is a regular local ring with maximal ideal $(\pi_j, \pi_k).$ 
 In particular, $\pi_j$ gives a discrete valuation on 
 $F_{m_i}$ which extends the discrete valuation on $F$ given by $\pi_j.$ 
 Let $F_{m_i, \pi_j}$ be the completion of $F_{m_i}$ at the 
 discrete valuation given by $\pi_j.$ 
 Then $F_{\pi_j} \subset F_{m_i, \pi_j}.$ 
 Let $\displaystyle L = F(\sqrt[n]{\pi_1\pi_2}, \sqrt[n]{\pi_2\pi_3}).$ 
 Suppose that $F$ contains $\rho,$ 
 a primitive $n^2$-th root of unity. 
\end{notation}

\begin{cor} 
\label{m_i}
With notations as in (\ref{three_maximal_ideals}), 
we have $T_{L/F}(F_{m_i}) = RT_{L/F}(F_{m_i})<\rho>.$
\end{cor}

\begin{proof}
Since $\pi_2$ is a unit at $m_2,$ 
we have $m_2A_{m_2} = (\pi_1\pi_2, \pi_3\pi_2).$ 
Hence, by (\ref{R-trivial_times_rho}), we have 
$\displaystyle T_{L/F}(F_{m_2}) = RT_{L/F}(F_{m_2})<\rho>.$ 
Since $\pi_1$ is a unit at $m_1,$ $\displaystyle m_1A_{m_1} =
 (\pi_1\pi_2, \pi_1^{-1} \pi_3).$ 
 Since $\displaystyle L = F(\sqrt[n]{\pi_1\pi_2}, \sqrt[n]{\pi_2 \pi_3}) 
 = \displaystyle F(\sqrt[n]{\pi_1\pi_2}, \sqrt[n]{\pi_1^{-1} \pi_3}) ,$ 
 by (\ref{R-trivial_times_rho}), 
 we have $T_{L/F}(F_{m_1}) = RT_{L/F}(F_{m_1})<\rho>.$ 
 Similarly, $T_{L/F}(F_{m_3}) = RT_{L/F}(F_{m_3})<\rho>.$
\end{proof}

\begin{cor} 
\label{pi_i}
With notations as in (\ref{three_maximal_ideals}), 
we have $T_{L/F}(F_{\pi_i}) = RT_{L/F}(F_{\pi_i})<\rho>.$
\end{cor}

\begin{proof}
Let $\kappa(\nu_{\pi_i})$ be the residue field of $F_{\pi_i}.$ 
The discrete valuation $\nu_{\pi_i}$ of $F$ given 
by $\pi_i$ has unique extension $\tilde{\nu}_{\pi_i}$ to $L$. 
Since $F$ contains a primitive $n^{\rm th}$ root of unity, 
the residue field $\kappa(\tilde{\nu}_{\pi_i})$ of $L$ at 
$\tilde{\nu}_{\pi_i}$ is a cyclic extension of $\kappa(\pi)$ of degree $n.$
 In particular, 
 $T_{\kappa(\tilde{\nu}_{\pi_i})/\kappa(\nu_{\pi_i})}( \kappa(\nu_{\pi_i})) = 
RT_{\kappa(\tilde{\nu}_{\pi_i})/\kappa(\nu_{\pi_i})}( \kappa(\nu_{\pi_i})).$
 Hence, by (\ref{residue_field_assumption}), 
 $$T_{L/F}(F_{\pi_i}) = RT_{L/F}(F_{\pi_i})<\rho>.$$
\end{proof}

\begin{cor}
\label{rho_notin_r1}
 Let $F_{m_i, \pi_j}$ be as in (\ref{three_maximal_ideals}). 
 Then $\rho^t \in RT_{L/F}(F_{m_i, \pi_j}) $ if and only if $n$ divides $t.$ 
\end{cor}

\begin{proof}
 Since the residue field of $F_{m_i, \pi_j}$ is a complete discretely valued field 
 with the image of $\pi_k$ ($k \neq i, j$) as a parameter 
 and the image of $\pi_i$ as a unit, 
 it is easy to see that 
 $\displaystyle L\otimes_F F_{m_i, \pi_j } \simeq F_{m_i, \pi_j}(\sqrt[n]{v\pi_j}, \sqrt[n]{u})$ 
 for some units $u$ and $v$ such that 
 $\displaystyle [ F_{m_i, \pi_j}(\sqrt[n]{u}) : F_{m_i, \pi_j}] = n.$ 
 Thus, the corollary follows from (\ref{rho_non-trivial}). 
\end{proof}

 For each $1 \leq i \neq j \leq 3,$ we have inclusions fields 
 $F_{m_i} \to F_{m_i, \pi_j}$ and $F_{\pi_j} \to F_{m_i, \pi_j}.$ 
 Thus we have the induced homororphisms 
 $\alpha_{ij} : T_{L/F}(F_{m_i})/R \to 
 T_{L/F}(F_{m_i, \pi_j})/R$ and $\beta_{ji} : 
 T_{L/F}(F_{\pi_j})/R \to T_{L/F}(F_{m_i, \pi_j})/R.$
 
\begin{lemma}
\label{not-onto}
 The product map
 $$\displaystyle \phi : ( \prod_{i = 1}^3 T_{L/F}(F_{m_i}) /R) 
 \times (\prod_{j =1}^3 T_{L/F}(F_{\pi_j})/R) 
 \to \prod_{1 \leq i \neq j \leq 3} 
 (T_{L/F}(F_{m_i, \pi_j})/R)$$
 is not onto. 
\end{lemma}

\begin{proof} 
 Let $y_{12} = \rho \in T_{L/F}(F_{m_1, \pi_2})$ 
 and $y_{ij} = 1 \in T_{L/F}(F_{m_i, \pi_j})$ 
 for all $i \neq j$ and $(i, j) \neq (1, 2).$ 
 Then we show that 
 $\displaystyle y = (y_{ij}) \in \prod_{1 \leq i \neq j \leq 3} 
 (T_{L/F}(F_{m_i, \pi_j})/R)$
  is not in the image of $\phi.$ 
 
 Suppose $y$ is in the image of $\phi.$ 
 Then there exist $a_i \in T_{L/F}(F_{m_i})$ and 
 $b_j\ \in T_{L/F}(F_{\pi_j})$ such that 
 $\phi(a_1, a_2, a_3, b_1, b_2, b_3) = y$ 
 modulo $R$-trivial elements. 
 Then we have 
 $\alpha_{12}(a_1) \beta_{21}(b_2) = y_{12} = \rho$ 
 modulo $R$-trivial elements and 
 $\alpha_{ij}(a_i) \beta_{ji}(b_j) \in RT_{L/F}(F_{m_i, \pi_j})$ 
 for all $i \neq j$ and $(i, j) \neq (1, 2).$ 
 By (\ref{m_i}) and (\ref{pi_i}), 
 we have $a_i = c_i \rho^{s_i}$ for some 
 $c_i \in RT_{L/F}(F_{m_i})$ and 
 $b_j = d_j \rho^{t_j}$ for some $d_i \in RT_{L/F}(F_{\pi_j}).$ 
 Hence $a_i = \rho^{s_i}$ and 
 $b_j = \rho^{t_j}$ modulo $R$-trivial elements. 
Since $\rho \in F,$ $\alpha_{ij}(\rho) = \rho$ and 
$\beta_{ji}(\rho) = \rho$ for all $i \neq j.$ 
We have $\rho = y_{12} = \alpha_{12}(a_1) \beta_{21}(b_2) = 
\rho^{s_1 + t_2}$ modulo $R$-trivial elements. 
Hence, by (\ref{rho_notin_r1}), $n$ divides $1 - s_1 - t_2.$

 Let $1 \leq i \neq j \leq 3$ with $(i, j) \neq (1, 2).$ 
 Then $ 1 = \alpha_{ij}(a_i) \beta_{ji}(b_j) = \rho^{s_i + t_j}$ 
 modulo $R$-trivial elements. 
 Hence $\rho^{s_i + t_j } \in RT_{L/F}(F_{m_i, \pi_j})$ 
 and by (\ref{rho_notin_r1}), $n$ divides $s_i + t_j.$ 
 Since $n$ divides $s_2 +t_1$ and 
 $s_3 + t_1,$ $n$ divides $s_3 - s_2.$ 
 Since $n$ divides $s_1 + t_3$ and 
 $s_2 + t_3,$ $n$ divides $s_1 - s_2.$ 
 Hence $n$ divides $s_1 - s_3.$ 
 Since $n$ divides $s_3 + t_2,$ $n$ divides $s_1 + t_2,$ 
 which contradicts the fact that $n$ divides $1 - s_1 - t_2.$ 
\end{proof}

\begin{theorem} 
\label{lgpfails} 
Let $K$ be a complete discretely valued field 
with residue field $\kappa$ and ring of integers $R.$ 
Let $\XX$ be a regular integral surface proper over $R$ 
and $F$ be its fraction field. 
Let $X$ denote the reduced special fibre of $\XX$. 
Suppose that $X$ is a union of regular curves with normal crossings. 
Suppose that there exist three three irreducible curves 
$X_1,$ $X_2$ and $X_3$ regular on $\XX$ such that 
$X_i \cap X_j,$ $i \neq j$ has exactly one closed point. 
Let $n \geq 2$ be an integer coprime to char$(\kappa).$ 
Suppose that $K$ has a primitive $n^2$-th root of unity. 
Then there exists a Galois extension $L/F$ of 
degree $n^2$ with Galois group isomorphic to 
$\Z/n\Z \times \Z/n\Z$ such that 
the local-global principle fails for $T_{L/F}$. 
\end{theorem}

\begin{proof} 
Let $P_1,$ $P_2$ and $P_3$ be the points of $X_i \cap X_j,$ $i \neq j.$ 
Let $A$ be the semi local ring at $P_1,$ $P_2$ and $P_3$ on $\XX.$ 
Then $A$ has three maximal ideals $m_1, m_2$ and $m_3.$ 
Since $\XX$ is regular and each $X_i$ is regular on 
$\XX,$ there exist primes $\pi_1, \pi_2, \pi_3 \in A$ such that 
$m_i = (\pi_j, \pi_k)$ for all distinct $ i, j , k.$ 
Let $\displaystyle L = F(\sqrt[n]{\pi_1\pi_2}, \sqrt[n]{\pi_2\pi_3}).$ 
Since $K$ contains primitive $n^{\rm th}$ root of unity, 
$L/F$ is a Galois extension with Galois group isomorphic to $\Z/n\Z \times \Z/n\Z.$ 
We claim that the local-global principle fails for $T_{L/F}$.

Let $\PP$ be a finite set of closed points of $X$ containing all the singular points of $X.$
 Let $X \setminus \PP = \cup U_i,$ with $U_i \subset X_i$ 
 for $i = 1, 2, 3.$ 
 By (\cite[Corollary 3.6.]{HHK1}), it is enough to show that 
 the product map $$\displaystyle  \prod_{P \in \PP} T_{L/F}(F_{P}) 
 \times \prod_i T_{L/F}(F_{U_i}) \to 
 \prod_{P, U_i } T_{L/F}(F_{P, U_i})$$ is not onto. 
 Since $X_1, X_2, X_3$ are the only curves in $X$ 
 passing through $P_1, P_2$ or $P_3,$ it is enough to show that 
 $$\displaystyle \phi : \prod_{i = 1}^3 T_{L/F}(F_{P_i}) 
 \times \prod_{j = 1}^3 T_{L/F}(F_{U_j}) 
 \to \prod_{P_i, U_j} T_{L/F}(F_{P_i, U_j})$$ 
 is not onto. 
 Since $F_{U_j} \subset F_{\pi_j}$ and 
 $F_{P_i, U_j} = F_{P_i, \pi_j},$ $\phi$ 
 factors as 
$$\displaystyle \prod_{i = 1}^3 T_{L/F}(F_{P_i}) 
\times \prod_{j = 1}^3 T_{L/F}(F_{U_j}) \to 
\prod_{i = 1}^3 T_{L/F}(F_{P_i}) 
\times \prod_{j = 1}^3 T_{L/F}(F_{\pi_j}) \to 
 \prod_{P_i, \pi_j} T_{L/F}(F_{P_i, \pi_j}).$$ 
 Since, by (\ref{not-onto}), 
 $$\displaystyle  \prod_{i = 1}^3 (T_{L/F}(F_{\pi_i})/R) 
 \times \prod_{j = 1}^3 (T_{L/F}(F_{P_j})/R) 
 \to \prod_{U_i, P_j} (T_{L/F}(F_{P_i, U_j})/R)$$
is not onto, $\phi$ is not onto. 
\end{proof}

\begin{remark} 
The above theorem for $\kappa$ algebraically closed and 
$n = 2$ is proved by 
Colliot-Th\'el\`ene, Parimala and Suresh 
(\cite[Section 3.1. \& Corollary 6.2.]{CTPS1}). 
\end{remark}

\begin{cor}
\label{norm1example}
Let $K$ be a complete discretely valued field with 
residue field $\kappa$ and ring of integers $R.$ 
Let $t \in R$ be a parameter. 
Let $\XX = \text{Proj}\left( R [x, y, z] /\langle xy(x+y-z) - tz^3\rangle \right).$
Let $X$ be the special fibre of $\XX.$ 
Then $X = \text{Proj} \left( \kappa[x, y, z]/\langle xy(x+y - z)\rangle \right)$ which is reduced. 
Then $X$ has three irreducible components $X_1,$ $X_2,$ $X_3$ and 
$X_i$ intersects $X_j,$ $i \neq j$ at exactly one point. 
Let $F$ be the function field of $\XX.$ 
Then $F \simeq K(x)[ y]/ \langle xy(x+y-1)-t\rangle.$ 
Let $n \geq 2$ be coprime to char$(\kappa).$ 
Suppose that $K$ contains a primitive $n^2$-th root of unity. 
Let $\displaystyle L = F(\sqrt[n]{xy}, \sqrt[n]{y(x-1)}).$ 
Then $L/F$ is a Galois extension with Galois group 
isomorphic to $\Z/n\Z \times \Z/n\Z.$ 
By (\ref{lgpfails}), the local-global principle for 
fails for $T_{L/F}$. 
\end{cor} 

\begin{proof} 
Let $U =$ Spec$R[x, y]/(xy(x+ y -1) - t).$ 
Then $U$ is an affine open subset of $\XX.$ 
Let $P_1 = (1, 0),$ $P_2 = (0, 1)$ and $P_3 = (0, 0)$ 
be the three closed points of $U.$ 
Let $A$ be the semi local ring at $P_1,$ $P_2$ and $P_3$ 
and let $m_i$ be the maximal ideal of $A$ corresponding to $P_i.$ 
Then $m_1 = (x + y -1, y),$ $m_2 = (x , x+ y -1)$ and $m_3 = (x, y).$ 
Hence, by (\ref{lgpfails}), the local-global principle fails for 
$T_{L/F}$.
\end{proof}

\begin{cor} 
Let $K$ be a complete discretely valued field with residue field $\kappa$ 
and ring of integers $R.$ Let $t \in R$ be a parameter. 
Let $\XX = $ Proj $\left( R [x, y, z] / \langle xy(x+y-z)(x-2z) - tz^4 \rangle \right)$ 
and $F$ be the function field of $\XX.$ 
Then $F \simeq K(x)[y]/(xy(x+y- 1)(x-2)-t).$ 
Let $\theta_1 = (x - 2 )/(x - 2 + xy(x +y -1))$
and $\theta_2 = (y -2) / (y - 2 + xy(x + y-1)).$ 
Let $n \geq 2$ with $6n$ coprime to char$(\kappa).$ 
Let $\displaystyle L_1 = F(\sqrt[n]{xy}, \sqrt[n]{y(x+y -1)})$ and 
$\displaystyle L_2 = F(\sqrt[n]{xy \theta_1 }, \sqrt[n]{y(x + y-1)\theta_2}).$ 
Then $L_1$ and $L_2$ are Galois extensions of $F$ that are 
linearly disjoint and the local-global principle fails for 
$T_{L_1 \times L_2/F}.$
\end{cor}

\begin{proof} 
To show that the local-global principle fails for 
$T_{L_1 \times L_2/F}$, by (\cite[Theorem 3.6.]{HHK}) 
and as in the proof of (\ref{lgpfails}), 
it is enough to show that 
$$\displaystyle  \phi : \prod_{i = 1}^3 (T_{L_1 \times L_2/F}(F_{\pi_i}) / R) 
\times \prod_{j = 1}^3 (T_{L_1 \times L_2/F}(F_{P_j})/ R) 
\to \prod_{U_i, P_j} (T_{L_1 \times L_2/F}(F_{P_i, U_j})/R)$$ is not onto. 

 Let $U =$ Spec$R[x, y]/(xy(x+ y -1)(x-2) - t).$ 
 Then $U$ is an affine open subset of $\XX.$ 
 Let $P_1 = (1, 0),$ $P_2 = (0, 1),$ $P_3 = (0, 0)$ and $Q = (2, 2).$ 
 Let $A$ be the semi local ring at $P_1,$ $P_2,$ $P_3$ and $Q.$ 
 Let $m_i$ be the maximal ideals of $A$ corresponding to $P_i$ 
 and $m$ be the maximal ideal corresponding to $Q.$ 
 Let $\pi_1 = x,$ $\pi_2 = y$ and $\pi_3 = x+y-1.$ 
 Then $m_1 = (\pi_2, \pi_3),$ 
 $m_2 = (\pi_1 , \pi_3),$ 
 $m_3 = (\pi_1, \pi_2).$ 
 We also have $m = (x-2, y-2).$
 Since $2 \neq $ char$(\kappa),$ 
 $x-2$ and $y -2$ are units at $m_i$ and 
 $\theta_i = 1$ modulo $m_j$ and $\pi_j.$ 
 Since $n$ is coprime to char$(\kappa),$ 
 $\theta_i \in F_{P_i}^n$ and $\theta_i \in F_{\pi_j}^n$ 
 for all $i$ and $j.$ 
 Hence $L_1 \otimes_F F_{\pi_i} \simeq L_2 \otimes_F F_{\pi_i}$ 
 and $L_1 \otimes_F F_{P_j} \simeq L_2 \otimes_F F_{P_j}.$ 
 By (\ref{product_of_field}), we have 
 $T_{L_1 \times L_2/F}(F_{\pi_i}) /R \simeq T_{L_1/F}(F_{\pi_i})$/R, 
 $T_{L_1 \times L_2/F}(F_{P_j}) /R \simeq T_{L_1/F}(F_{P_j})/R$ and 
 $T_{L_1 \times L_2/F}(F_{U_i, P_j}) /R \simeq T_{L_1/F}(F_{U_i, P_j})/R.$ 
 Since, by (\ref{not-onto}), 
 $$\displaystyle  \prod_{i = 1}^3 (T_{L/F}(F_{\pi_i})/R) 
 \times \prod_{j = 1}^3 (T_{L/F}(F_{P_j})/R) 
 \to \prod_{U_i, P_j} (T_{L/F}(F_{P_i, U_j})/R)$$ 
 is not onto, $\phi$ is not onto. 
 Hence the local-global principle fails for $T_{L/F}$.

Since $\pi_1\pi_2 = xy = 4$ modulo $m$ 
and $\pi_2 \pi_3 = 6 $ modulo $m,$ 
we have $\displaystyle L_1 \otimes_F F_Q = F_Q(\sqrt[n]{4}, \sqrt[n]{6}).$ 
Since $6n$ is coprime to 
char$(\kappa),$ $L_1 \otimes_F F_Q/F_Q$ is unramified. 
Since $xy(x+y -1) = 12$ modulo $m,$ $x-2 + xy(x+y-1) = 12 a^n$ 
for some $a \in F_Q.$ Similarly, $y - 2 + xy(x+y+1) = 12 b^n$ 
for some $b \in F_Q.$ 
Hence $\displaystyle L_2 \otimes_F F_Q = F_Q(\sqrt[n]{(x-2)/3}, \sqrt[n]{(y-2)/2}).$ 
Since the maximal ideal $m = (x-2, y-2)$ and $3, 2$ are units at $m,$ 
$L_1 \otimes_F F_Q$ and $L_2 \otimes_F F_Q$ are 
linearly disjoint over $F_Q.$ In particular, 
$L_1$ and $L_2$ are linearly disjoint over $F.$
\end{proof}

\vspace{0.2in}

\textbf{Acknowledgements}:
The author thanks his advisor Prof. V. Suresh for all his
help and guidance during this work. 
The author is also very thankful to Prof. R. Parimala 
for several helpful discussions. The author thanks 
Professors J.-L. Colliot-Th\'el\`ene and R. Parimala 
for going through an earlier version of the manuscript 
and providing their helpful feedbacks. 
The author thanks the anonymous referee for the 
constructive feedback due to which the exposition has 
improved considerably. 
The author is partially supported by National Science Foundation grants 
DMS-1463882 and DMS-1801951. 
The author thanks Emory University 
and Laney Graduate School for 
excellent working conditions.

\providecommand{\bysame}{\leavevmode\hbox to3em{\hrulefill}\thinspace}

\end{document}